\documentclass{amsart}

\usepackage{amsmath, amssymb, amsthm, amsfonts, tikz,mathtools, enumerate, graphicx}
\usepackage{wasysym}
\usepackage[alphabetic]{amsrefs}
\usepackage{hyperref}
\usepackage{cleveref}
\usepackage[margin=1.25in]{geometry}
\usepackage{fancybox}
\usepackage{calc}
\usetikzlibrary{cd}

\newtheorem*{theorem*}{Theorem}
\newtheorem{theorem}{Theorem}[section]
\newtheorem{lemma}[theorem]{Lemma}
\newtheorem{proposition}[theorem]{Proposition}

\theoremstyle{definition}
\newtheorem{definition}[theorem]{Definition}

\newtheorem{remark}[theorem]{Remark}
\newtheorem{example}[theorem]{Example}

\numberwithin{equation}{section}

\newcommand{\cs}{\ensuremath{C^*}} 
\newcommand{\adj}{\ensuremath{\mathcal{L}}}
\newcommand{\bK}{\mathbb{K}}

\newcommand{\C}{\mathbb{C}}
\newcommand{\Q}{\mathcal{Q}}
\newcommand{\R}{\mathcal{R}}
\newcommand{\id}{\text{id}}

\newcommand{\supp}{\text{supp}}
\newcommand{\inv}{^{-1}}

\makeatletter
\newcommand\thmsname{\protect\theoremname}
\newcommand\nm@thmtype{theorem}
\theoremstyle{plain}

\newenvironment{namedthm}[1][Undefined Theorem Name]{
  \ifx{#1}{Undefined Theorem Name}\renewcommand\nm@thmtype{theorem*}
  \else\renewcommand\thmsname{#1}\renewcommand\nm@thmtype{namedtheorem}
  \fi
  \begin{\nm@thmtype}}
  {\end{\nm@thmtype}}
\theoremstyle{plain}

\title{Principal Actions on Topological Quivers and Associated Operator Dynamics}

\author[Gillespie]{Matthew Gillespie}
\address{Matthew Gillespie
\\School of Mathematical and Statistical Sciences\\
Arizona State University\\
Tempe, AZ 85287}
\email{}

\author[Hall]{Lucas Hall}
\address{Lucas Hall
\\Department of Mathematics
\\Michigan State University
\\East Lansing, Michigan 48824}
\email{hallluc1@msu.edu}

\author[Jones]{Benjamin Jones}
\address{Benjamin Jones
\\School of Mathematical and Statistical Sciences\\
Arizona State University\\
Tempe, AZ 85287}
\email{brjone16@asu.edu}

\author[Tobolski]{Mariusz Tobolski}
\address{Mariusz Tobolski
\\Instytut Matematyczny, Uniwersytet Wroc\l{}awski, 
pl. Grunwaldzki 2/4, 50-384 Wroc\l{}aw, Poland
}
\email{mariusz.tobolski@math.uni.wroc.pl}

\begin{document}
\begin{abstract}
    We study topological quivers $Q$ admitting a free and proper action by a locally compact group $G$ together with their associated \cs-algebras. On the topological side, we provide a complete classification of topological quivers which admit such actions in terms of $G$-bundles over the vertex orbit space and an appropriate isomorphism of bundles over the edge orbits. Following the work by Deaconu, Kumjian, and Quigg on topological graphs, we construct an isomorphism between $C^*(Q/G)$ and Rieffel's fixed-point algebra $C^*(Q)^\alpha$, which is known to be Morita equivalent to $C^*(Q)\rtimes_rG$. Unlike the previously mentioned work with topological graphs, we use previously developed functoriality techniques to identify the isomorphism. We also examine many concrete examples of such group actions, including some exclusive to topological quivers, and the associated Morita equivalences.
\end{abstract}
\maketitle

\section{Introduction}

Since their introduction in 1997 \cite{kprr}, Graph (\cs-)algebras have enjoyed a great deal of attention in the field of operator algebras. The reasons for this are many, but we will mention here only three. Firstly, they offer multiple realizations: models coming from dynamics; generators and relations; and \cs-correspondences. Secondly, features of the graph's combinatorial data lucidly determine algebraic features of the associated graph algebra. Lastly, these features extend widely to more generalized contexts (see \cite{graphalgebras} for a well respected overview, and specifically the last two chapters for two divergent generalizations, namely those of higher-rank graphs and of topological graphs). The second reason alone is incredibly appealing, since it yields connections across numerous fields of mathematics, and indeed graph algebras have both borrowed techniques from these other fields and inspired progress in others. As just one example, graph algebras spawned the purely algebraic study of Leavitt Path algebras, introduced in 2005 \cite{aaLeavitt} and independently \cite{amp}, has now garnered enough attention to earn  at least one standard reference \cite{BookLeavitt}, as well as published conference proceedings \cite{MR4377762}.  

One example of borrowed techniques comes in the study of dynamics of directed graphs and their associated \cs-algebras. In 1999, Kumjian and Pask demonstrated that a discrete group $G$ acting on a directed graph $E$ induces an action on the associated graph \cs-algebra $\cs(E)$, and whenever the action is free, there is a Morita equivalence between the quotient graph \cs-algebra $\cs(E/G)$ and the reduced crossed product $\cs(E)\rtimes_r G$ \cite{kpaction}. Shortly after in \cite{kqrskew}, Kaliszewski, Quigg, and Raeburn leveraged a classification of free actions on graphs \cite{grosstucker} through a \emph{skew product} construction to construct \emph{coactions} on the graph \cs-algebras, and using the theory of nonabelian duality recovered Kumjian and Pask's original result (Raeburn's \cite{graphalgebras}*{Chapter 6} offers a good overview of this stream of discovery). 

In the intervening years, attention has steadily turned to the investigation of generalized dynamics (here meaning the agnostic study of actions and coactions of classical groups on \cs-algebras) amidst the aforementioned generalizations to graph algebras. Along the direction of higher rank graphs \cite{graphalgebras}*{Chapter 10} and their generalizations are the works \cites{pqr, pqs, fps, kkqs,  bkq} which concern generalized dynamics of higher-rank graphs. We mention here that the discrete nature of higher-rank graphs lends itself to a higher rank version of the Gross-Tucker classification of graphs admitting free group actions \cite{pqr}*{Corollary 6.5} and lends itself once again to a natural and universal description of dynamics in terms of coactions and nonabelian duality. However, this article is not concerned with this research direction. 

Instead, this article takes inspiration from Katsura's suite of articles on topological graphs \cites{katsura1, katsura2, katsura3, katsura4}, and the generalized dynamics of (in general, locally compact) groups on these topological objects \cites{dkq, kqskew}. Historically, this pair of papers followed the pattern set out for directed graph algebras, investigating actions of locally compact groups $G$ on topological graphs $E$ and the associated algebras $\cs(E)$ before turning attention to skew product topological graphs directly. It is foundational to note that in the topological setting there are obstructions to a classification of free actions discussed above for the directed and higher rank graph settings, since general locally compact spaces keep track of local gluing phenomena which are not present for discrete spaces. One gains some ground when they consider at least principal actions on the topological space (see Theorem \ref{thmhus}), however unlike the discrete setting a skew product construction will not model all possible dynamics (cf. Examples \ref{qr} and \ref{kickass}). Nevertheless, in \cite{dkq} the authors identify conditions under which one may classify dynamics on topological graphs as an inheritance from the pullback construction which does not feature cocycles, and in \cite{kqskew} the authors analyze this structure in the presence of a cocycle, recover the associated coaction construction, and consequentially the Morita equivalence result using nonabelian duality. This marks a departure from the discrete setting, where either method is available to produce the Morita Equivalence-- topological graphs which admit a free and proper action are not always witnessed as skew product graphs. 

We mention that \cite{dkq} inspired the sequence of papers \cites{kqrfunctor, kqrCPcoaction, bkqr} which introduced valuable categorical tools to investigate functoriality of suitable correspondence homomorphisms between \cs-correspondences $(X, A)$, and importantly to induce homomorphisms between associated Cuntz-Pimsner algebras $\mathcal{O}_X.$ In the first article \cite{kqrfunctor}, the authors mention how to apply these general results for Cuntz-Pimsner algebras to the case of a topological graph, and \cite{kqskew} directly uses the classification of skew product topological graphs from \cite{dkq} together with the developments from \cite{kqrCPcoaction} to construct a coaction analogous to that which is found in \cite{kqrskew}. The tools from \cites{kqrfunctor, kqrCPcoaction, bkqr} are very influential to this article as well.

Like the pair of companion articles \cites{kpaction, kqrskew} and \cite{dkq, kqskew}, which study the actions and coactions of groups on directed graphs and topological graphs respectively, this article is the group action companion to the second author's investigation of coactions on topological quivers \cite{hskew}. Topological quivers \cites{mstensor, mtquiver} serve as the widest topological generalization to directed graphs, subsuming topological graphs and including rudimentary examples excluded by the restrictions on topological graphs; the plane $\mathbb{R}^2$ with range and source maps the coordinate projections comes to mind. The cost of this generality is the abandonment of local homeomorphisms and the introduction of new tools to measure the volume of edges emitted by a vertex (a role previously played in private by counting measures for both directed and topological graphs). The associated payoff is a class of examples which exhibits new phenomena for examining the topological--operator algebraic parallel. For example, \cite{bmaximal}*{Theorem 6.1 and Corollary 6.2} offers a recent display that topological quivers often violate Arveson's Hyperrigidity conjecture, recently disproven in \cite{b-do}. By comparison, correspondences of topological graphs always satisfy the hyperrigidity conjecture \cite{bmaximal}*{Corollary 6.2}. In this article, we investigate principal group actions on topological quivers, and introduce many new examples which exhibit the necessity of studying actions of locally compact groups in the full generality of topological quivers. We then employ the techniques of \cite{kqrfunctor} and \cite{dkq} to confirm features of topological graphs and their \cs-algebras analogous to those found for directed and topological graphs. On the topological side, we offer the following classification of principal actions on topological quivers. 
\begin{namedthm}[Theorem A](Theorem \ref{Topologicalclassification})
Given a topological quiver $Q$, a principal $G$-bundle $p:P\to E^0$ over the vertices of $Q,$ and an isomorphism of pullbacks $s^*(P)\cong r^*(P),$ there is a topological quiver $R$ with vertex space $P$ and edge space $s^*(P)$ which admits a principal action by $G.$  This construction accounts for every topological quiver which admits a principal $G$-action. 
\end{namedthm}
From the operator algebraic framework, topological quivers offer a natural example of a ``noncommutative topological space'' underlying the algebraic structure of a \cs-algebra, which is popularly regarded in the contemporary paradigm as the algebra of continuous functions on such a noncommutative topological space. As in the classical setting of a group on a topological space, a group action on a quiver induces an action on the associated \cs-algebra, and we have the following structure result for principal actions relating the reduced crossed product to the quiver algebra of the quotient quiver. 
\begin{namedthm}[Theorem B](Theorem \ref{ME})
If a locally compact group $G$ acts principally on a topological quiver $Q,$ then $\cs(Q)\rtimes_r G$ and $\cs(Q/G)$ are strongly Morita equivalent. 
\end{namedthm}

Throughout, a significant innovation is the maintenance of certain families of measures on the edge space of the topological quiver, conditions which appear in Lemmas \ref{convergenceofcrosssections} and \ref{convergence} and carry on into the main results. We also generalize a handful of constructions to the setting of quivers. Compared with directed or topological graphs, these measures are a new feature for topological quivers, though they have also appeared in similar contexts such as topological groupoids, where they play a role in defining the multiplication by convolution \cite{toolkit}. We also provide a large class of examples which displays the flexibility of topological quivers and the wide range of admissible group actions. These examples range from classical operator algebras associated to the the discrete setting of directed graphs all the way to the continuous setting which is only accessible with the full generality of topological quivers. 
% We introduce a novel class of topological quivers from coset spaces of locally compact Hausdroff groups. This class of quiver C*-algebras includes many known examples obtainable using directed and topological graphs, e.g., the Toeplitz algebra, the irrational rotation algebra, and the Cuntz algebras. Moreover, we provide genuine topological quiver examples, for instance, the coaction crossed product of $\mathcal{O}_\infty$ by the orthogonal group $O(n)$ (see Example~\ref{ex3.14}). Furthermore, we show that topological group relation quivers~\cite{mccann} always admit a free and proper action.

The article is organized as follows: Section \ref{prelim} reviews the topological and operator algebraic preliminaries concerning topological quivers, $G$-bundles, and Cuntz-Pimsner algebras, and the tools from \cite{kqrfunctor}. Section \ref{actionexample} defines actions on topological quivers and introduces some classes of examples together with numerous instances which distinguish the examples within the known literature. Section \ref{topaspects} completely classifies the topological quivers which admit principal $G$-actions, and section \ref{correspondences} translates this classification into an operator algebraic framework. To finish, we present Morita equivalences for some of the examples provided in section \ref{actionexample}, particularly for examples requiring the full generality of topological quivers. 

\subsection{Acknowledgements}
This research is part of the EU Staff Exchange project 101086394
“Operator Algebras That One Can See” and was partially supported by the University of Warsaw Thematic Research Programme ``Quantum Symmetries". 
The second and fourth authors are also grateful for the hospitality of IMPAN during a visit in February 2024. 
The second author also recognizes the Zuckerman Institute, who has supported him as a Postdoctoral scholar during the completion of this work.
The fourth author is also delighted to thank Arizona State University in Tempe for their great hospitality.

\section{Preliminaries}\label{prelim}

\subsection{Quivers and Bundles}

\begin{definition}
A \emph{bundle} is nothing more than a surjective continuous map $p:E\to B$. We call the domain $E$ the \emph{total space} and the codomain $B$ the base space in this context. Additionally, one refers to the closed subspace $p\inv(v)\subset E$ as the \emph{fibre} over the base-point $v$. If the total space $E$ admits an action by $G$ (that is, a continuous map $\alpha:E\times G\to E$ which respects the group multiplication $\alpha\circ (\alpha\times \id) = \alpha \circ (\id\times m)$), then $p$ is called a \emph{$G$-bundle} (over $B$) provided there is an isomorphism $B\cong E/G$ of the base space and the quotient/orbit space of the action. A $G$-bundle is called \emph{principal} provided the action on $E$ is both free and proper. 
\end{definition}
Let us state the following result which follows from~\cite[Theorem~4.2]{husemoller}. This result was originally proven in \cite{palais}. See also \cite{hkqw2}*{Theorem 3.8} for a direct proof in more contemporary language which holds for actions of locally compact groupoids.
\begin{theorem}\label{thmhus}
Suppose we are given principal $G$-bundles $q_{i}:P_i\rightarrow X_i\cong P_{i}/G$
for $i=1,2$ and a pair of continuous maps $f:X_{1}\rightarrow X_{2}$,
$\widetilde{f}:P_{1}\rightarrow P_{2}$ such that $\widetilde{f}$
is $G$-equivariant and $q_{2}\circ\widetilde{f}=f\circ q_{1}$. Then $\theta_{f}:P_{1}\rightarrow f^{*}(P_{2})$
given by $\theta_{f}(p)=(q_{1}(p),\widetilde{f}(p))$ is an isomorphism
of $G$-bundles.
\end{theorem}

The interpretation of Theorem \ref{thmhus} is that the only way to get a $G$-equivariant map between total spaces of principal $G$-bundles is the pullback construction.  

A \emph{topological quiver} is an ensemble $(E^0, E^1, r, s, \lambda)$ comprising second countable locally compact spaces $E^i$ together with continuous maps $r,s:E^1\to E^0$, where $s$ is additionally open, and a continuous full and proper $s$-system of measures $\lambda = \{\lambda_v\}_{v\in E^0}$. The last structure entails that $\lambda_v$ are measures on $E^1, \lambda_v \neq 0,$ $\supp(\lambda_v) = s\inv(v),$ and that given any $\xi\in C_c(E^1),$ the assignment
\[v\mapsto \int_{E^1} \xi(e)\,d\lambda_v(e)\]
determines a continuous function. Notice when $\supp(\xi) = K$, this function is always supported on a closed subset of $s(K)$. These systems of measures are precisely the tools mentioned in the introduction which measure the ``volume'' of edges emitted from a vertex. When the source map is a local homeomorphism (i.e. the quiver is a topological graph in the sense of Katsura \cite{katsura1}) the $s$-system is populated by counting measures and safely omitted from the data. By extension, when the vertex space $E^0$ and edge space $E^1$ are discrete topological quivers instantiate directed graphs.

Let $Q=(E^0, E^1, r_Q, s_Q, \lambda)$ and $R=(F^0, F^1, r_R, s_R, \mu)$ be topological quivers. A \emph{morphism} of topological quivers is a pair of continuous open maps $e:E^1\to F^1$ and $v:E^0\to F^0$ such that $r_R\circ e = v\circ r_Q$, $s_R\circ e = v\circ s_Q$ and the pullback measures of $\mu$ under $e$ agree with the systems associated to $Q$. An isomorphism of topological quivers is a quiver morphism $(e, v)$ with $e, v$ homeomorphisms which also preserves the measurements of the $s_Q,$ (resp. $s_R$) systems (which is to say that the pushforward and pullback measures by the appropriate continuous maps assign the same values to the images and preimages of measureable sets). Then an automorphism of $Q$ is a quiver isomorphism from $Q$ to $Q$, and the family of all such forms a group under the natural compositions.

\subsection{Operator Algebras}

Our references for \cs-correspondences are \cites{raeburnwilliams, lance}. Homomorphisms between \cs-algebras are assumed to be $*$-preserving. The multiplier algebra of a \cs-algebra $A$ is denoted $M(A)$, and is the largest \cs-algebra containing $A$ as an essential ideal; in this article, the most natural realization for the multipliers is $M(A)=\adj(A)$, the set of all adjointable operators from $A$ to itself considered as a right Hilbert $A$-module over itself.   A homomorphism $\pi: A\to M(B)$ is \emph{nondegenerate} provided $\text{span}(\pi(A)B)$ is dense in $B$. Recall that any given nondegenerate homomorphism admits a unique extension $\bar{\pi}:M(A)\to M(B);$ there is a bit of fuss about composing degenerate homomorphisms, but at least when all of the maps are nondegnerate there is little risk of confusion. For more details one may refer to the appendix of \cite{boiler}. We will have occasion to work with degenerate homomorphisms. We work often with tensor products, and move freely among such isomorphisms as $A\otimes C_0(T) = C_0(T, A)$ for locally compact Hausdorff $T.$ 

 A \cs-correspondence is a triple $(B, X, A)$ comprising \cs-algebras $A, B$, a right Hilbert $A$-module $X$, and a left action of $B$ on $X$ formally given by nondegenerate homomorphism $\phi_B:B\to \adj(X).$ It is common to suppress the homomorphism, writing $b\cdot \xi$ instead of $\phi_B(b)(\xi)$, and we will follow this convention whenever convenient- retaining the notation $\phi_B$ to discuss compositions of homomorphisms. When $A=B,$ we further suppress notation and speak of the correspondence $(X, A)$. A \emph{homomorphism} of \cs-correspondences is a triple $(\rho, \psi, \pi):(B, X, A)\to (D, Y, C)$ where $\rho:B\to D$ and $\pi:A\to C$ are homomorphisms of the left and right coefficient algebras, and $\psi:X\to Y$ is a linear map compatible with the various correspondence operations. A correspondence homomorphism is \emph{nondegenerate} provided $\overline{\text{span} }(\psi(X)C)=Y$ and the coefficient homomorphisms are nondegenerate. Any \cs-algebra is a correspondence over itself, so it make sense to talk about a correspondence homomorphism into a \cs-algebra. Given any correspondence $(X, A)$, there is a \cs-algebra $\mathcal{T}_X$ which is universal for correspondence homomorphisms into it: a correspondence homomorphism into any other \cs-algebra $D$ induces a unique $*$-homomorphism from $\mathcal{T}_X$ into $D$.  

Given any correspondence $(B, X, A)$ it's \emph{multiplier correspondence} is the triple $(M(B), M(X), M(A))$, where $M(A)$ and $M(B)$ are the multiplier algebras, $M(X) = \adj(A,X)$ is the adjointable operators from $A$ to $X$, where $A$ is viewed as the \cs-module over itself, and the operations are given by pre/post composition. It is not hard to see that there is an inclusion $(B, X, A)\subseteq (M(B), M(X), M(A))$, and moreover any nondegenerate correspondence homomorphism $(\rho, \psi, \pi):(B, X, A)\to (D, Y, C)$ admits a unique extension to the multiplier correspondence. In addition to the algebra of all adjointable operators $\mathcal{L}(X)$ there is the \cs-subalgebra of \emph{compact operators} 
\[\mathbb{K}(X) = \overline{\text{span}}\{\theta_{\eta, \xi}\in \mathcal{L}(X): \xi,\eta\in X \text{ and } \theta_{\eta, \xi}(x) = \eta\cdot \langle \xi, x\rangle (a)\}.\]
One has $M(\mathbb{K}(X)) = \mathcal{L}(X)$. 

We now introduce the following objects associated to the correspondence homomorphism $( \psi, \pi):(X, A)\to (Y ,B)$ and an ideal $J\subset A$
\begin{align*}
M_B(Y)&=\{a\in \mathcal{L}(Y): aB\cup Ba\subset Y \}\\
M(B, J)&=\{a\in M(B): aJ\cup Ja\subset B\}\\
\psi^{(1)}&:\mathbb{K}(X)\to \mathcal{L}(Y); \psi^{(1)}(\theta_{\eta, \xi}) =\theta_{\psi(\eta), \psi(\xi)} 
\end{align*}
Regarding the last map $\psi^{(1)},$ we mention that the \emph{Cuntz-Pimsner} algebra of the correspondences $(X, A)$ is the quotient of $\mathcal{T}_X$ by the ideal generated by 
\[\{\psi^{(1)}(\phi_A(a))-i_A(a): a\in J_X\},\]
where $i_A$ is the inclusion map and $J_X$ is the \emph{Katsura ideal,} the largest ideal of $A$ on which the left action is an injective map into the compact operators. 

The following definition is essential to this work.
\begin{definition}[\cite{kqrfunctor}*{Definition 3.1}]\label{CPcovcor}
A correspondence homomorphism is \emph{Cuntz-Pimsner covariant} provided 
\begin{enumerate}[(i)]
\item $\psi(X)\subset M_B(Y)$
\item $\pi:A\to M(B)$ is nondegenerate
\item $\pi(J_X)\subset M(B; J_Y)$
\item The following diagram commutes 
\end{enumerate}
\begin{equation*}
\begin{tikzcd}
J_X\ar["\pi|"]{r}\ar["\phi_A|", swap]{d} & M(B; J_Y)\ar["\overline{\phi_B}|"]{d}\\
\bK(X)\ar["\psi^{(1)}", swap]{r} & M_B(\bK(Y)).\\
\end{tikzcd}
\end{equation*}
\end{definition}

Recall that an \emph{action} of a locally compact group $G$ on a \cs-algebra $A$ is a function $\alpha:G\times A\to A$, written $\alpha_t(a)=\alpha(t, a),$ such that for fixed $a\in A$, the function $t\mapsto \alpha_t(a)$ is continuous. Equivalently, an action may be described by a nondegenerate homomorphism $\widetilde{\alpha}:A\to M(A\otimes C_0(G))$ so that the diagram 
\begin{center}
\begin{tikzcd}
A \ar["\widetilde{\alpha}"]{r}\ar["\widetilde{\alpha}"]{d}& M(A\otimes C_0(G))\ar["\widetilde{\alpha}\otimes \id"]{d}\\
M(A\otimes C_0(G))\ar["\id\otimes m", swap]{r} & M(A\otimes C_0(G)\otimes C_0(G))
\end{tikzcd}
\end{center}
commutes, where $m:C_0(G)\to C_0(G)\otimes C_0(G)$ is described by $m(f)(s,t) = f(st)$. An action of $G$ on a correspondence $(X, A)$ comprises an action $\alpha$ on $A$, together with a map $\gamma:G\times X\to X$ such that for fixed $\xi\in X,$ the map $t\mapsto \gamma_t(\xi)$ is continuous, and for all $t, (\gamma_t, \alpha_t)$ is an invertible correspondence homomorphism.  

We will apply all of the above in the context of topological quivers, which we introduced immediately after Theorem \ref{thmhus}. Given a topological quiver $Q=(E^0, E^1, r, s, \lambda),$ one constructs a \cs-correspondences over $A=C_0(E^0)$ as follows: Set $X_0 = C_c(E^1)$, which is a bimodule over $A$ with left and right actions defined for $\xi\in X_0$ and $f\in A$ by 
\[(\xi f)(e) = \xi(e)f(s(e)) \qquad (f\xi)(e) = f(r(e))\xi(e).\] $X_0$ also admits an $A$-valued inner product given for $\xi, \eta\in X_0$ by 
\[\langle \eta, \xi \rangle(v) = \int \overline{\eta}(e)\xi(e)\,d\lambda_v(e), \]
and $X$ is the completion of $X_0$ with respect to the norm $\|\xi\|^2 = \|\langle \xi, \xi\rangle\|_A$. The \cs-algebra of a topological quiver $\cs(Q)$ is by definition the Cuntz-Pimsner algebra of this correspondence. 

\section{Actions on Quivers and Examples}\label{actionexample}

Let us introduce the notion of a group action on a topological quiver. This also appears in \cite{hskew}.
\begin{definition}
Let $Q = (E_0, E^1, r, s, \lambda)$ be a topological quiver. A locally compact group $G$ acts on $Q$ provided 
there are group actions on the vertex and edge spaces for which the range and source maps are both $G$-equivariant, and invariance for the $s$-system of measures in the following sense (\cite{toolkit}*{Section 3.2}): the given $G$-action on $E^1$ induces an action on $C_c(E^1)$ according to $\phi(e)\cdot g = \phi(e\cdot g),$ and we require that 
\[\int \phi(e\cdot g)\,d\lambda_{v}(e) = \int \phi(e)\,d\lambda_{v\cdot g}(e)\]
(where $\lambda_{v\cdot g}(E)$ is defined by $\lambda_v(E\cdot g)$) whenever $\phi\in C_c(E^1), e\in E^1, v\in E^0, $ and $g\in G.$ Such a group action is called \emph{principal} provided the action is principal on the vertex space. 
\end{definition}
It follows (see Theorem \ref{thmhus}) that such an action is also principal on the edge space. Said another way, a group action on a quiver $Q$ is a continuous homomorphism from $G$ to the group of automorphisms of $Q$.

\begin{definition}
    If $Q=\left(E^0, E^1, r, s, \lambda\right)$ is a topological quiver on which a second-countable locally compact group $G$ acts with quotient maps $q^0:E^0 \rightarrow E^0/G$ and $q^1:E^1 \rightarrow E^1/G$, then define the quotient quiver $Q/G=\left(E^0/G, E^1/G, \dot{r}, \dot{s}, \dot{\lambda} \right)$ such that for $f\in C_c\left(E^1/G\right)$,
    \[
    \dot{r}\left(q^1(e)\right)=q^0\left(r(e)\right),\qquad \dot{s}\left(q^1(e)\right)=q^0\left(s(e)\right)
    \]
    \[
    \int f(u)d\dot{\lambda}_{q^0(v)}(u) = \int (f\circ q^1)(e)d\lambda_v(e)
    \]
\end{definition}

Let $G$ be a group action on a topological quiver $Q = (E^0, E^1, r, s, \lambda)$. Whenever we have a free and proper action of $G$ on a quiver $Q$, we may form the quotient quiver $Q/G$ by \cite{hskew}*{Proposition 3.6}. 

\begin{example}
Let $Q=(E^0, E^1, r, s, \lambda)$ be a topological quiver, $G$ a second countable locally compact group, and $\kappa:E^1\to G$ any continuous function. The \emph{skew product quiver} is the ensemble 
\begin{align*}
Q\times_\kappa G &= (E^0\times G, E^1\times G, r_\times, s_\times, \lambda_\times )\\
s_\times(e, g) = (s(e),g)\qquad r_\times(e, g) = (r(e), &g\kappa(e)), \qquad \int_{E^1\times G}\xi(e, g)\,d\lambda_{\times v, g}(e, g) = \int_{E^1}\xi(e, g)\,d\lambda_{v}(e).
\end{align*}
In \cite{hskew}, the second author showed that skew product quivers are indeed topological quivers, and provides a refinement of Theorem \ref{Topologicalclassification} parallel to \cite{dkq}*{Corollary 3.7} in the case where the $G$-bundle over the vertex space is trivial. Skew product quivers admit principal actions by $G,$ and using techniques from nonabelian duality the second author deduced the Morita Equivalence \ref{ME} as a consequence of Katayama duality \cite{hskew}*{Corollary 5.3}. Throughout this section we introduce numerous concrete examples, including those which leave the context of skew product quivers. 
\end{example}

\subsection{Topological quivers coming from coset spaces}

Let $G$ be a locally compact topological group and $N$ be a closed subgroup. The quotient map $\pi_N:G\to N\backslash G$ is open and $N\backslash G$ has a group structure whenever $N$ is normal. 
%Indeed, if $U$ is open in $G$, then $\pi(U)$ is open in $G/H$ if and only if $\pi^{-1}(\pi(U))=UH=\bigcup_{h\in H}Uh$ is open in $G$, which is immediate. 
Furthermore, since $N$ is closed in $G$, one can show that the function
\[
Ng\longmapsto \int_N\xi(ng)d\mu_N(n)
\]
belongs to $C_c(N\backslash G)$ for all $\xi\in C_c(G)$. Here $\mu_N$ is the right-invariant Haar measure on $N$. %Furthermore, every function in $C_c(G/N)$ is of this form. 
Therefore, we can define the following topological quiver.

\begin{definition}\label{cosetquiver}
    Let $G$ be a second-countable locally compact group, 
    let $N$ be a closed subgroup of $G$, and let $\phi:G\to G$ 
    be a continuous map. We define the {\em (right) coset quiver} $Q_{N< G}^\phi$ by the data 
    \[
Q_{N< G}^\phi = (N\backslash G, G,\;s,\;r,\;\lambda)
    \]
    according to the rules 
    \begin{align*}
 s(g):=Ng,\qquad & r(g):=N\phi(g),\\
   \lambda = \{\lambda_{Ng}\}_{Ng\in G/N},\qquad \int_{Ng} \xi(ng)d\lambda_{Ng}(ng)&:=\int_N\xi(ng)d\mu_N(n),\qquad \xi\in C_c(G),
    \end{align*}
    where $\mu_N$ is the right-invariant Haar measure on $N$. If $\phi(g)=gk$ for a fixed $k\in G$, we denote the coset quiver by $Q_{N<G}^k$. 
\end{definition}
\noindent Left coset quivers are defined analogously. For $H$ a closed subgroup, the quotient map $\pi_H:G\to G/H$ is open and the function
\[
gH\longmapsto \int_H \xi(gh)d\mu_H(h)
\]
belongs to $C_c(G/H)$ for all $\xi\in C_c(H)$, where $\mu_H$ is the left-invariant Haar measure on $H$. In the case of a closed normal subgroup, left and right coset quivers coincide.

As rudimentary examples, we recognize the Toeplitz algebra, the irrational rotation algebra, and the Cuntz algebras as Cuntz--Pimsner algebras of coset quivers.
\begin{example}[Toeplitz algebra]\label{toealg}
Let $G=\mathbb{Z}_2=\{1,-1\}$, let $N=\{1\}$, and let $\phi(1)=\phi(-1)=1$. Then $C^*(Q_{\{1\}<\mathbb{Z}_2}^\phi)\cong\mathcal{T}$, where $\mathcal{T}$ is the Toeplitz algebra, i.e. the unital $C^*$-algebra generated by an isometry.
\end{example}
\begin{example}[Irrational rotation algebra]
Let $G=\mathbb{T}$, let $N=\{1_\mathbb{T}\}$, and let $\phi(z)=e^{2\pi\theta i}z$ for all $z\in\mathbb{T}$, where $\theta\in(0,1)$ is irrational. Then the Cuntz-Pimsner algebra of $Q_{\{1\}<\mathbb{T}}^{\exp(2\pi i \theta) }$ is the irrational rotation algebra $A_\theta=C(\mathbb{T})\rtimes\mathbb{Z}$. In general, if $G$ is a locally compact group, $N=\{1_G\}$, and $\phi:G\to G$ is a homeomorphism, then $C^*(Q_{\{1_G\}<G}^\phi)\cong C_0(G)\rtimes_\phi\mathbb{Z}$ unless $\phi^n={\rm id}_G$ for some $n\in\mathbb{N}$. In the latter case, $C^*(Q_{\{1_G\}<G}^\phi)\cong C_0(G)\rtimes_\phi\mathbb{Z}_n$.

%Next, if $H=\mathbb{T}$, then the induced action of $\mathbb{T}$ on $C^*(Q(\mathbb{T},\{1_\mathbb{T}\},e^{2\pi\theta i}))$ comes from the action of $\mathbb{T}$ on $C(\mathbb{T})$.
\end{example}
\begin{example}[Cuntz algebras] \label{oinfty}
    Let $G$ be a second-countable locally compact group and $N=G$ and $\phi(g)=g$ for all $g\in G$ in Definition~\ref{cosetquiver}. We get the following quiver
\[
Q_{G<G}^\phi = (\{1_G\}, G, r, s, \mu_G )
\]
where $r, s$ are constant maps and $\mu_G$ is the right-invariant Haar measure on $G$. We obtain $A_Q=C(\{1_G\})=\mathbb{C}$ and $X_Q=L^2(G)$. Since $G$ is second countable, $L^2(G)$ is separable. Therefore, the Cuntz-Pimsner algebra $C^*(Q_{G<G}^\phi)$ is isomorphic to $\mathcal{O}_{|G|}$ for finite $G$ and $\mathcal{O}_\infty$, i.e. the Cuntz algebra generated by countably infinitely many isometries, for infinite $G$.
\end{example}

In Example~\ref{toealg}, the only action of $\mathbb{Z}_2$ on $Q_{\{1\}<\mathbb{Z}_2}^\phi$ for which $\phi$ is equivariant is the action by multiplication on the edge space and the trivial action on the vertex space. However, the next proposition shows that coset quivers can often be equipped with free actions.
\begin{proposition}\label{cosetaction}
Let $G$, $N$, and $\phi$ be as in Definition~\ref{cosetquiver} and let $H$ be a closed subgroup of $G$ such that $Ng\cap gH=\{g\}$ for all $g\in G$. The formulas
\begin{equation}\label{hactquiv}
    (g,h)\longmapsto gh,\qquad (Ng,h)\longmapsto Ngh,
\end{equation}
define a continuous and free actions of $H$ on $G$ and $N\backslash G$. If $\phi$ is $H$-equivariant, the formulas~\eqref{hactquiv} induce a continuous and free action of $H$ on the coset quiver $Q_{N<G}^\phi$.
\end{proposition}
Notice that the intersection condition for $H$ and $N$ is guaranteed if $N$ is a normal subgroup of $G$ and $N\cap H=\{e\}$. 
\begin{proof}
First, it is clear that the formulas~\eqref{hactquiv} define continuous actions, since multiplication is continuous, and that the $H$-action on the edge space $G$ is free. Concerning freeness of the $H$-action on the vertices $N\backslash G$, this follows from the assumption $Ng\cap gH=\{g\}$ for all $g\in G$ -- indeed, 
if $Ngh=Ng$, then there exist $n,n'\in N$ such that $ngh=n'g$. In turn, $gh=n^{-1}n'g\in Ng\cap gH$, which implies that $h=1_H$ and satisfies the first claim. 

When $\phi$ is $H$-equivariant, note that
\[
s(gh)=Ngh=(Ng)h=s(g)h\quad\text{and}\quad
r(gh)=N\phi(gh)=N\phi(g)h=(N\phi(g))h=r(g)h
\]
so the source and the range map are $H$-equivariant. Finally, we check whether the system $\{\lambda_{Ng}\}_{Ng\in N\backslash G}$ is $H$-invariant. To this end, note that if $\xi\in C_c(G)$, then denoting the right translation of $h\in H$ by $\rho_h$ we have $((\rho_h)(\xi))(g):=\xi(gh)$ for all $h\in H$ and $\rho_g(\xi)\in C_c(G)$. Since
\begin{align*}
\int_{Ng}\xi(ng\cdot h)d\lambda_{Ng}(ng)&=\int_{Ng}((\rho_h)(\xi))(ng)d\lambda_{Ng}(ng)=\int_N((\rho_h)(\xi))(ng)d\mu_N(n)\\
&=\int_{N}\xi(ngh)d\mu_N(n)=\int_{Ngh}\xi(ngh)d\lambda_{Ngh}(ngh),
\end{align*}
we obtain a continuous action of $H$ on the topological quiver $Q(G,N,\phi)$.
\end{proof}
\noindent Observe that our setup is symmetric, i.e. one could start with a left coset quiver using a closed subgroup $H$ and then for a closed subgroup $N$ such that $Ng\cap gH=\{g\}$ for all $g\in G$ there is a continuous free action given by the formulas
\begin{equation}\label{leftcoset}
(n,g)\longmapsto ng,\qquad (n,gH)\longmapsto ngH,\qquad g\in G,~n\in N.
\end{equation}
If $H$ is a closed subgroup of a topological group $G$, then its action on $G$ by right multiplication is always proper (see~\cite{cartan}) and 
the $H$-action on $N\backslash G$ is proper whenever $N=\{1_G\}$ or $H$ is compact. 
However, our next example shows that the action of $H$ on $N\backslash G$ described by Proposition~\ref{cosetaction} might fail to be proper.
\begin{example}
Consider the group extension
\[
0\longrightarrow \mathbb{Z}\longrightarrow\mathbb{R}\overset{p}{\longrightarrow}\mathbb{T}\longrightarrow 0,
\]
where the first homomorphism is the inclusion and $p(x):=e^{2\pi xi}$. Pick any irrational $\theta\in [0,1]$ and  consider the closed subgroup $H=\theta\mathbb{Z}$ of $\mathbb{R}$. Since $\theta\mathbb{Z}\cap\mathbb{Z}=\{0\}$, there is a free action of $\theta\mathbb{Z}$ on both $\mathbb{R}$ and $\mathbb{T}$ by Proposition~\ref{cosetaction} but the action of $\theta\mathbb{Z}$ on $\mathbb{T}$ is not proper (actions of locally compact non-compact groups on compact spaces cannot be proper). 
\end{example}

Equipped now with both a quiver and a free action on it by a group $H$ we may now study the quotient quiver. Let $G, N, \phi$ and $H$ satisfy the hypotheses of Proposition~\ref{cosetaction}, meaning $N$ and $H$ are subgroups of $G$, and $H$ acts on $Q_{N<G}^\phi$. Assume also that the $H$-action on this quiver is proper and let $\pi_N:G\to N\backslash G$ and $\pi_H:G\to G/H$ denote the quotient maps. Clearly, $\pi_N$ is $H$-equivariant and $\pi_H$ is $N$-equivariant for the $N$-actions on $G$ and $G/H$ given by the formulas~\eqref{leftcoset}.
Next, note that $(N\backslash G)/H\cong N\backslash(G/H)$ by the map
\[
(N\backslash G)/H\longrightarrow N\backslash(G/H),\qquad (Ng)H\longmapsto N(gH).
\] 
%Therefore, the properness of the $H$-action on $N\backslash G$ implies the properness of the $N$-action on $G/H$ by~\cite{}. 
The spaces $(N\backslash G)/H$ and $N\backslash(G/H)$ are also homeomorphic to the double coset space $N\backslash G/H$.
The actions~\eqref{hactquiv} and~\eqref{leftcoset} give rise to the quotient maps
\[
\pi_{NH}:N\backslash G\longrightarrow N\backslash G/ H,\qquad Ng\longmapsto NgH,
\]
\[
\pi_{HN}:G/H\longrightarrow N\backslash G/H,\qquad gH\longmapsto NgH,
\]
and we have the following commuting diagram
\[
\begin{tikzcd}
G\arrow[r,"\pi_H"] \arrow[d,"\pi_N"'] & G/H \arrow[d,"\pi_{HN}"]\\
N\backslash G\arrow[r,"\pi_{NH}", swap] & N\backslash G/H.
\end{tikzcd}
\]
Any $H$-equivariant $\phi:G\to G$ induces a continuous $\dot{\phi}:G/H\to G/H,$ with $\pi_H\circ \phi = \dot{\phi}\circ \pi_H$ and we also have the following commuting diagram
\[
\begin{tikzcd}
G\arrow[r,"\pi_H"] \arrow[d,"\pi_N\circ\phi"'] & G/H \arrow[d,"\pi_{HN}\circ\dot{\phi}"]\\
N\backslash G\arrow[r,"\pi_{NH}", swap] & N\backslash G/H.
\end{tikzcd}
\]
Indeed, by $H$-equivariance of $\phi$, we have that
\[
\pi_{HN}\circ\dot{\phi}\circ\pi_H=\pi_{HN}\circ\pi_H\circ\phi=\pi_{NH}\circ\pi_{N}\circ\phi.
\]
The quotient quiver of $Q_{N<G}^\phi$ by $H$ is given by $(N\backslash G/H, G/H, \pi_{HN}\circ\dot{\phi}, \pi_{HN},\dot{\lambda})$, where
\[
\int_{\{ngH~:~n\in N\}}\xi(ngH)d\dot{\lambda}_{NgH}(ngH):=\int_N\xi(\pi_H(ng))d\lambda_{Ng}(ng)=\int_N\xi(\pi_H(ng))d\mu_N(n).
\]

\subsection{Coset quivers of semi-direct products and skew-products}
We now specialize our study of coset quivers to the setting of normal subgroups and semidirect product groups. If $G$ is a group with a normal subgroup $N$ and a subgroup $H$ such that $N\cap H=\{e\}$ and $NH=G$, then $G$ is isomorphic to the algebraic semi-direct product $N\rtimes H$ via the isomorphism $(n,h)\mapsto nh$~\cite{robinson}. 
Here the action of $H$ on $N$ is given by $n\mapsto hnh^{-1}$. If $G$ is a locally compact and $N$ and $H$ are closed subgroups, this algebraic isomorphism is a continuous bijection but it may fail to be open. 
However, it is a homeomorphism whenever $N\rtimes H$ is $\sigma$-compact \cite{hewittross}*{Theorem 5.29} (e.g. when $N$ and $H$ are second-countable). 
Since $N$ is normal, $N\cap H=\{e\}$, and $NH=G$, we obtain that $N\backslash G=G/N\cong H$. 
Hence, whenever $\phi$ is $H$-equivariant, the $H$-action on $Q_{N<G}^\phi$ given by Proposition~\ref{cosetaction} is the right-multiplication on the edges $G$ and on the vertices $H$, which is always proper.
The next results show that, under additional conditions on $\phi$, the coset quiver of a semi-direct product of groups can be realized as a skew-product quiver.
\begin{lemma}\label{semiskewlemma}
Let $G$ be a second-countable locally compact group 
and let $N$ and $H$ be closed subgroups of $G$ such that $N$ is normal, $N\cap H=\{e\}$, and $NH=G$. 
If $\phi(g)=c(n)h$ for every $g=nh\in G$, where $c:N\to H$ is a continuous map, then there is an isomorphism of topological quivers
\[
Q_{N<G}^\phi \cong Q_{N<N}^e\times_c H.
\]
Furthermore, if we equip $Q_{N<G}^\phi$ with the $H$-action given by Propostion~\ref{cosetaction} and the skew-product quiver $Q_{N<N}^e\times_c H$ with the canonical $H$-action, the isomorphism is $H$-equivariant.
\end{lemma}
\begin{proof}
First, note that $\phi(nh)=c(n)h$ is a well-defined continuous map. Indeed, $\phi$ is the composition of the continuous inverse to $(n,h)\mapsto nh$, the product $c\times{\rm id}_H:N\times H\to H\times H$, and the multiplication map of $G$.
Define the following maps
\[
\psi^1:(Q_{N<N}^e\times_c H)^1=N\times H\longrightarrow G=(Q_{N<G}^\phi)^1,\quad \psi^1(n,h)=nh,
\]
\[
\psi^0:(Q_{N<N}^e\times_c H)^0=\{1_N\}\times H \longrightarrow N\backslash G=(Q_{N<G}^\phi)^0,\quad \psi^0(1_N,h)=Nh.
\]
Since $G\cong N\rtimes H$, in light of the preceeding discussion we infer that both $\psi^1$ and $\psi^0$ are homeomorphisms. If we denote the source and range maps of the skew product by $s_\times$ and $r_\times$, respectively, then
    \[
\psi^0(s_\times(n,h))=\psi^0(1_N,h)=Nh=s(nh)=s(\psi^1(n,h)),
    \]
    \[
\psi^0(r_\times(n,h))=\psi^0(1_N,c(n)h)=Nc(n)h=N\phi(nh)=r(nh)=r(\psi^1(n,h)).
    \]
    Next, let $\{\lambda^\times_{(1_N,h)}\}_{h\in H}$ denote the system of measures on the skew-product. We have that:
    \begin{align*}
        \int_{Nh}\xi(nh)d\lambda_{Nh}(nh)&=\int_N\xi(nh)d\mu_N(n)=\int_N \xi(\psi^1(n,h))d\mu_N(n)\\
        &=\int_N\xi(\psi^1(n,h))d\lambda_{1_N}(n)=\int_{\{(n,h)~:~n\in N\}}\xi(\psi^1(n,h))d\lambda^\times_{(1_N,h)}(n,h),
    \end{align*}
    where $\xi\in C_c(G)$ and $ \mu_N$ is the right-invariant Haar measure on $N$. So the morphisms preserve measures and we have a quiver isomorphism.  Finally, we check whether $\psi^1$ and $\psi^0$ are $H$-equivariant:
    \[
    \psi^1((n,h)h')=\psi^1(n,hh')=nhh'=(nh)h'=\psi^1(n,h)h',
    \]
    \[
\psi^0((1_N,h)h')=\psi^0(1_N,hh')=Nhh'=(Nh)h'=\psi^0(1_N,h)h'.
    \]
    Therefore, we obtain an $H$-equivariant isomorphism of topological quivers.
\end{proof}
The above lemma implies the following result.
\begin{proposition}\label{semiskew}
    Let $G$, $N$, $H$, and $\phi$ be as in Lemma~\ref{semiskewlemma}. The quiver C*-algebra $C^*(Q_{N<G}^k)$ is isomorphic to the coaction-crossed-product \cs-algebra $C^*(Q_{N<N}^e)\rtimes_\delta H$, where the coaction $\delta$ is obtained from $\phi$ as in~\cite{hskew}. %Furthermore, the action induced on $C^*(Q(G,N,k))$ induced by the $H$-action given in Proposition~\ref{cosetaction} is the dual $H$-action on the coaction-crossed-product C*-algebra.
\end{proposition}

\begin{example}[Dihedral groups]\label{dihedral}
Let $D_n=\langle a,b~|~a^2=b^n=1,~aba=b^{-1}\rangle$ be the dihedral group, i.e. the group of symmetries of the regular $n$-gon, $n\geq 1$. From this presentation of $D_n$ it is clear that $D_n\cong\mathbb{Z}_n\rtimes\mathbb{Z}_2$. Since all groups in sight are finite, by Proposition~\ref{semiskew} and Example~\ref{oinfty}, we obtain isomorphisms
\[
C^*(Q_{\mathbb{Z}_2< D_n}^a)\cong C^*(Q_{\mathbb{Z}_2<\mathbb{Z}_2}^1)\rtimes_\delta\mathbb{Z}_2\cong\mathcal{O}_n\rtimes_\delta\mathbb{Z}_2\,,\qquad n\geq 1.
\]
%It is worth noting that, since all the above algebras are Cuntz--Krieger algebras, one can use~\cite{} to compute their K-theory groups. 
%By~\cite{} and~\cite{}, one can conclude that $C^*(Q(D_n,\mathbb{Z}_n,a))$ is isomorphic to $\mathcal{O}_2$ for all $n\geq 1$. However, these algebras are different when treated as $\mathbb{Z}_2$-algebras, where we consider the dual $\mathbb{Z}_2$-action.
\end{example}
\begin{example}[Orthogonal groups]\label{ex3.14}
    Let $O(n)$ be the orthogonal group in dimension $n\geq 2$. It is known that $O(n)\cong SO(n)\rtimes\mathbb{Z}_2$, where $\mathbb{Z}_2=\{1,-1\}$ is identified with $\{I,R\}$, where $I$ is the identity matrix and $R={\rm diag}(1,-1,\ldots,-1)$. Let $\phi:O(n)\to O(n)$ be defined by the formula $\phi(M)=MR$. The source and the range maps read 
    \[
    s(M)=\det(M)\qquad\text{and}\qquad r(M)=SO(n)MR.
    \]
    Lemma~\ref{semiskewlemma} and Proposition~\ref{semiskew} apply so that we have the following isomorphisms:
    \[
    C^*(Q_{SO(n)<O(n)}^R)\cong C^*(Q_{SO(n)<SO(n)}^I)\rtimes_\delta\mathbb{Z}_2\cong\mathcal{O}_\infty\rtimes_\delta\mathbb{Z}_2.
    \]
\end{example}

\begin{example}[Euclidean groups]\label{euclidean}
Let $E(n)$ be the group of isometries of the $n$-dimensional Euclidean space. It is well-known that $E(n)\cong \mathbb{R}^n\rtimes O(n)$, where $\mathbb{R}^n$ and $O(n)$ correspond to the subgroups of translations and rotations, respectively. If $k\in O(n)$, then
\[
C^*(Q_{\mathbb{R}^n<E(n)}^k)\cong C^*(Q_{\mathbb{R}^n<\mathbb{R}^n}^1)\rtimes_\delta O(n)\cong \mathcal{O}_\infty\rtimes_\delta O(n).
\]
Since $Q_{\mathbb{R}^n<\mathbb{R}^n}^0$ is a bonefide topological quiver, the coset quivers associated to Euclidean groups are genuine examples of quivers which cannot be recovered as topological graphs. Notice also that this coaction cannot be witness as an action of the dual group because $O(n)$ is nonabelian.
\end{example}

The following example demonstrates that a coset quiver need not occur as a skew product. However, this fine example also shows a connection to skew products reminiscent of local trivality in the topological setting. 

\begin{example}\label{qr}
Consider the nontrivial principal $\mathbb{Z}$-bundle 
\[p:\mathbb{R}\to \mathbb{T}\qquad x\mapsto e^{2\pi i x},\]
the $\mathbb{Z}$-equivariant homeomorphisms
\begin{align*}
\phi:\mathbb{R}\to \mathbb{R}&\qquad x\mapsto x+1/2\\
\dot{\phi}:\mathbb{T}\to \mathbb{T}&\qquad z\mapsto -z,
\end{align*}
and then the coset quivers $Q_\mathbb{R}:= Q_{\{0\}<\mathbb{R}}^\phi$ and $Q_\mathbb{T}:= Q_{\{1\}<\mathbb{T}}^{\dot{\phi}}$. Because the normal subgroup is the trivial group, the each of the source maps is a homeomorphism and both examples are topological graphs. $Q_\mathbb{R}$ admits a principal $\mathbb{Z}$-action for which the associated quotient quiver is $Q_\mathbb{T}$, and in fact the surjective $\mathbb{Z}$-equivariant quiver morphism is given by $(p,p)$. Since $p$ is a nontrivial $\mathbb{Z}$-bundle over $\mathbb{T}$, we infer that $Q_\mathbb{R}$ is not a skew product quiver coming from \cite{hskew}, however we will soon see that $Q_\mathbb{R}$ is not far off from being a skew product.  

Consider the following cover of $\mathbb{T}$ which is labeled according to the cardinal directions: 
\begin{align*}
N=\{(z_0,z_1)~:~z_1>0\},&\qquad W=\{(z_0,z_1)~:~z_0>0\},\\
S=\{(z_0,z_1)~:~z_1<0\},&\qquad E=\{(z_0,z_1)~:~z_0<0\}.
\end{align*}
Denoting by $U_0=N\cup S = \mathbb{T}\setminus \mathbb{R}$ the union of the boreal and austral semicircles (with eastern and western antipodal poles omitted), and similarly $U_1 = E\cup W = \mathbb{T}\setminus i\mathbb{R}$ the union of the oriental and occidental semicircles (with the northern and southern antipodal poles removed), we see that the set $\{U_0, U_1\}$ determines an open trivializing cover of $\mathbb{T}$ invariant under $\dot{\phi}$. Then one may easily check that each of $Q_{U_i} = (U_i, U_i, \id, \dot{\phi}, \lambda)$ is a topological quiver which we'll regard (intuitively but imprecisely) as a ``subquiver'' with the data inherited from $Q_\mathbb{T}$. 

Let's follow a similar procedure to define an open trivializing cover of $\mathbb{R}$ invariant under $\phi.$ First, identify the following open intervals of $\mathbb{R}$: 
\[N_0 = (0, 1/2),\quad  S_0 =  (1/2, 1),\quad  E_0 = (-1/4, 1/4),  \text{ and }W_0 = (1/4, 3/4).\]
We will also denote left translation of these sets by labels exemplified by $N_k = (0+k, 1/2+k)$. We have the following set equalities 
\begin{align*}
p^{-1}(U_0)&=\bigsqcup_{k\in\mathbb{Z}}N_k\sqcup S_k,\\
p^{-1}(U_1)&=\bigsqcup_{k\in\mathbb{Z}}E_k\sqcup W_k,
\end{align*}
together with homeomorphisms
\begin{align*}
\Psi_0:p^{-1}(U_0)&\longrightarrow U_0\times\mathbb{Z},\qquad \Psi_0(x)=(p(x),n)\quad\text{for $x\in N_n\sqcup S_n$},\\
\Psi_1:p^{-1}(U_1)&\longrightarrow U_1\times\mathbb{Z},\qquad \Psi_1(x)=(p(x),n)\quad\text{for $x\in E_n\sqcup W_n$}.
\end{align*}
As in the case above concerning $Q_{U_i}$ and $Q_\mathbb{T}$, one easily checks that $Q_{p\inv(U_i)} = (p\inv(U_i), p\inv(U_i), \id, \phi, \lambda)$ is a topological quiver with the maps inherited from $Q_\mathbb{R}$ by restriction. 

Next, define two continuous maps 
\begin{align*}
c_0:U_0&=N_0\sqcup S_0\longrightarrow\mathbb{Z},\qquad c_0(z)=\begin{cases}
    0 & z\in N_0,\\
    1 & z\in S_0,
\end{cases}\\
c_1:U_1&=E_0\sqcup W_0\longrightarrow\mathbb{Z},\qquad c_1(z)=\begin{cases}
    0 & z\in E_0,\\
    1 & z\in W_0,
\end{cases}
\end{align*}
and regard each $c_i$ as a cocycle into $\mathbb{Z}$ defined on the edge space of $Q_{U_i}$. Then we may define the skew product quivers 
\[Q_{U_i}\times_{c_i}\mathbb{Z} = (U_i\times \mathbb{Z}, U_i\times \mathbb{Z}, r_\times, s_\times, \lambda_\times),\]
where the source maps $s_\times$ are the identity and the range maps are given by $r_\times(e, n) = (-e,n+c_i(e)).$

We claim that there is a quiver isomorphism between $Q_{p\inv(U_i)}$ and $Q_{U_i}\times _{c_i}\mathbb{Z}.$ Indeed, for $i=0,$ let $x\in N_n\sqcup S_n$, and we check separately that if $x\in N_n$, then
\[
\Psi_0(r_E(x))=\Psi_0(x+1/2)=(p(x+1/2),n)=(-p(x),n+c_0(p(x)))=r_\times(p(x),n))=r_\times(\Psi_0(x)).
\]
Likewise, if $x\in S_n$, then
\[
\Psi_0(r_E(x))=\Psi_0(x+1/2)=(p(x+1/2),n+1)=(-p(x),n+c_0(p(x)))=r_\times(p(x),n)=r_\times(\Psi_0(x)).
\]
Since the computations with concerning the source maps are trivial (the sources are identity maps) and the measures involved are all point masses, $(\Psi_0, \Psi_0)$ is the desired quiver isomorphism.  Analogous computations hold for $c_1$ and $(\Psi_1, \Psi_1)$. Thus, even though there is a topological obstruction keeping $Q_\mathbb{R}$ from being a skew product quiver, $Q_\mathbb{R}$ is still ``covered'' locally by the skew products $Q_{p\inv(U_i)}$. We also mention that since $\phi$ and $\dot{\phi} $ are homeomorphisms, we have 
\[\cs(Q_\mathbb{R})\cong C_0(R)\rtimes_\phi \mathbb{Z} \quad \text{and} \quad \cs(Q_\mathbb{T})\cong C(\mathbb{T})\rtimes_{\dot\phi}\mathbb{Z}_2. \]

\end{example}

As a final offering of this subsection, we'll produce a genuine topological quiver which admits a principal action by a locally compact group $H$ and for which the vertex action is not trivial. Thus, this coset quiver is not a skew product, and can only be studied using the techniques introduced in this article. 

\begin{example}\label{kickass}
Take any two nondiscrete Lie groups $K$ and $N,$ so that $N$ admits an action by $K$, let $G$ be the semidirect product, $\phi:G\to G$ a continuous map, and $Q_{N<G}^\phi$ the associated coset quiver. Notice that topologically, $G=N\times K$. Since this construction works for any topological group $K,$ we may arrange for $K$ to admit an action by another nondiscrete group $H$ so that $K\to K/H$ is a nontrivial bundle. Since the only matter of interest concerning $G$ is it's topology, the action of $H$ on $K$ lifts to a (principal, nontrivial) action on $G$. Thus,the action of $H$ will determine an action on the coset quiver $Q_{N<G}^\phi$ provided that $\phi$ is also equivariant for the $H$-action. This is easily accomplished: one may take, for example, any map $\phi':N\to N$ and take $\phi = \phi'\times \id,$ although surely many other choices are available. 
Now, because the map $K\to K/H$ is a nontrivial bundle, the associated quiver cannot be witnessed as a skew product by $H$. Moreover, the source map for $Q_{N<G}^\phi$ is not a local homeomorphism because the fibres, each homeomorphic with $N$, are not discrete. Consequently the quiver under investigation is a quiver which is not a topological graph. 
The associated quotient quiver is $(K/H, N\times (K/H), \id,\dot{\phi}, \lambda )$ where the measures are still taken from the haar measures on $N.$ 
\end{example}

\subsection{Non-split extensions of groups} It is well known that semi-direct products of the form $G=N\rtimes H$ correspond to split group extensions~(e.g. see~\cite[Lemma~7.20]{rotman}), i.e. short exact sequences of groups
\[
1\longrightarrow N\longrightarrow G\longrightarrow G/N\longrightarrow 1
\]
that split. If $N$ is a closed normal subgroup of a second-countable locally compact group $G$, then we always obtain a group extension as above. Although this extension might not split, it will give rise to a coset quiver $Q_{N<G}^\phi$ for any continuous $\phi:G\to G$.  

\begin{example}[Binary octahedral group]
    Let $1$, $i$, $j$, and $k$ be the quaternionic units. The {\em binary octahedral group} $2O$ is a multiplicative group of order $48$ with the following elements
    \[
    \pm 1,\pm i,\pm j,\pm k,\qquad \text{($8$ elements)}\qquad
    \frac{\pm 1\pm i\pm j\pm k}{2},\qquad\text{($16$ elements)}
    \]
    \[
    \frac{\pm 1\pm i}{\sqrt{2}},\frac{\pm 1\pm j}{\sqrt{2}},\frac{\pm 1\pm k}{\sqrt{2}},\frac{\pm i\pm j}{\sqrt{2}},\frac{\pm i\pm k}{\sqrt{2}},\frac{\pm j\pm k}{\sqrt{2}}.\qquad\text{($24$ elements)}
    \]
    There is a well-known non-split extension
    \[
    1\longrightarrow \mathbb{Z}_2\longrightarrow 2O\longrightarrow S_4\longrightarrow 1,
    \]
where $\mathbb{Z}_2$ is the normal subgroup $\{1,-1\}$ of $2O$ and the quotient $2O/\mathbb{Z}_2$ is isomorphic to the symmetric group $S_4$. In other words, $2O$ is a non-trivial double cover of $S_4$. It is straightforward to verify that the element
\[
\frac{-1-i-j-k}{2}
\]
generates a subgroup isomorphic with $\mathbb{Z}_3$. Note also that
\[
\mathbb{Z}_2\cap\mathbb{Z}_3=\{1,-1\}\cap \left\{1,\frac{-1-i-j-k}{2},\frac{-1+i+j+k}{2}\right\}=\{1\}.
\]
Therefore, by Proposition~\ref{cosetaction}, $\mathbb{Z}_3$ acts freely on $S_4$. This action is automatically proper since $\mathbb{Z}_3$ is finite. 

There are many coset quivers associated with the above non-split extension. 
\begin{enumerate}
    \item First, we consider $Q(2O,\mathbb{Z}_2,k)$, where $k$ is a quaternionic unit. It is clear that $Q(2O,\mathbb{Z}_2,k)$ has 12 connected components, each isomorphic to $Q(D_2,\mathbb{Z}_2,a)\cong Q(\mathbb{Z}_2,\mathbb{Z}_2,1)\times_c\mathbb{Z}_2$ (see Example~\ref{dihedral}).
    Therefore, We conclude that $C^*(Q(2O,\mathbb{Z}_2,k))\cong \bigoplus_{i=1}^{12}\mathcal{O}_2\rtimes_\delta\mathbb{Z}_2$. As mentioned above, there is a free and proper action of $\mathbb{Z}_3$ on $Q(2O,\mathbb{Z}_2,k)$ which moves a given connected component to a different connected component and has $4$ orbits. Much in the same way, the induced action on $C^*(Q(2O,\mathbb{Z}_2,k))$ operates on the direct summands.
    \item Next, consider
    \[
    Q\left(2O,\mathbb{Z}_2,\frac{-1-i-j-k}{2}\right)
    \]
    which has $8$ connected components, each isomorphic to $Q(D_3,\mathbb{Z}_2,a)\cong Q(\mathbb{Z}_3,\mathbb{Z}_3,1)\times_c\mathbb{Z}_2$.
\end{enumerate}
Next, we consider a different non-split extension
\[
0\longrightarrow Q_8\longrightarrow 2O\longrightarrow S_3\longrightarrow 0,
\]
where $Q_8$ is the quaternion group isomorphic to the subgroup $\{\pm1,\pm i,\pm j,\pm k\}$ of $2O$ such that $2O/Q_8\cong S_3$. Again, $Q_8\cap\mathbb{Z}_3=\{1\}$, so $\mathbb{Z}_3$ acts freely and properly on $S_3$.
\end{example}

%\[
%1\to \mathbb{Z}_2\to SU(2)\to SO(3)\to 1
%\]

There is also the binary icosahedral group $2I$, which is another example of a nonsplit extension of the icosahedral group. Both  $2I$ and $2O$ are important discrete groups for studying quantum error correction \cite{ktexotic}. These examples are discrete analogs of the more general spin groups, which are the universal covers of the special orthogonal groups $SO(n)$. 

\subsection{Topological group relations}
In~\cite{mccann}, the author introduced the notion of a topological group relation quiver that we now recall.
\begin{definition}
Let $G$ be a second-countable locally compact group and let $\alpha$ and $\beta$ be continuous endomorphisms of $G$. We define the {\em topological group relation quiver} $Q_{\alpha,\beta}(G)$ as follows
\begin{align*}
    Q^1_{\alpha,\beta}(G)&:=\{(g,h)\in G\times G~:~\alpha(g)=\beta(h)\},\qquad Q^0_{\alpha,\beta}(G):=G,\\
s(g,h):=g,\qquad &r(g,h):=h,\qquad \int\xi(g,h)d\lambda_v(g,h):=\int_{\ker \beta}\xi(v,ky)d\mu(k),
\end{align*}
where $\mu$ is a right-invariant Haar measure on $\ker \beta$ and $y \in \beta^{-1}(\alpha(v))$. If $\beta^{-1}(\alpha(v))=\emptyset$, then we define $\lambda_v=0$, since in that case $s^{-1}(v)=\emptyset$.
\end{definition}
\noindent Because $s^{-1}(v)=\{g\}\times \beta^{-1}(\alpha(v))$ and $\beta^{-1}(\alpha(v))=h\ker\beta$ for any $h\in \beta^{-1}(\alpha(v))$, the system of measures $\{\lambda_v\}_{v\in G}$ is well defined.

Next, we construct a continuous, free, and proper action on $Q_{\alpha,\beta}(G)$. Consider the subset
\begin{equation}\label{relsubgrp}
A:=\{g\in G~:~\alpha(g)=\beta(g)\}\subseteq G.
\end{equation}
We claim that $A$ is a closed subgroup of $G$. Indeed, the following computations
\[
\alpha(gh)=\alpha(g)\alpha(h)=\beta(g)\beta(h)=\beta(gh),\qquad g,h\in A,
\]
\[
\alpha(g^{-1})=\alpha(g)^{-1}=\beta(g)^{-1}=\beta(g^{-1}),\qquad g\in A,
\]
show that $A$ is a subgroup (here we used the fact that $\alpha$ and $\beta$ are homomorphisms). Furthermore, continuity of $\alpha$ and $\beta$ implies that $A$ is closed.

\begin{proposition}
    Let $G$ be a second-countable locally compact group, let $Q_{\alpha,\beta}(G)$ be its topological group relation quiver, and let $A$ be the closed subgroup of $G$ given by~\eqref{relsubgrp}. The formulas
    \[
    ((g,h),a)\longmapsto (ga,ha),\qquad (g,a)\longmapsto ga,
    \]
    define a continuous, free, and proper action of $A$ on $Q_{\alpha,\beta}(G)$.
\end{proposition}
\begin{proof}
First, since $A$ is a closed subgroup of $G$, the action of $A$ on $Q^0_{\alpha,\beta}(G)$ is continuous, free, and proper. Next, the $A$-action on $Q^1_{\alpha,\beta}(G)$ is well defined. Indeed, for any $a\in A$ and $(g,h)\in Q^1_{\alpha,\beta}(G)$, we have that
\[
\alpha(ga)=\alpha(g)\alpha(a)=\beta(h)\beta(a)=\beta(ha).
\]
Assume that $(ga,ha)=(g,h)$. Then $ga=g$ and $ha=h$, which implies that $a=e$. Hence, the $A$-action on the edge space is free. To see that it is also proper,  note that the action of $G$ on $G\times G$ given by the formula $((g,h),a)\mapsto (ga,ha)$ is proper by~\cite[Proposition~1.3.3]{palais}. Then, from~\cite[Proposition~1.3.1]{palais} we infer that the restriction of this action to an action of a closed subgroup $A$ on an $A$-invariant closed subspace $Q^1_{\alpha,\beta}(G)$  is again proper. Furthermore, the system $\{\lambda_v\}_{v\in G}$ is $A$-invariant:
\[
\int\xi(ga,ha)d\lambda_v(g,h)=\int_{\ker\beta}\xi(va,kya)d\mu_G(k)=\int_{s^{-1}(va)}\xi(g,h)d\lambda_{va}(g,h).
\]
Here we used the fact that $ya\in \beta^{-1}(\alpha(va))$ for every
$y\in\beta^{-1}(\alpha(v))$. Finally, the source and the range maps are $A$-equivariant so we obtain a continuous, free, and proper $A$-action on $Q_{\alpha,\beta}(G)$.
\end{proof}

\begin{example}
Let $G=\mathbb{T}$. Every continuous endomorphism of $\mathbb{T}$ is of the form $z\mapsto z^n$ for some $n\in\mathbb{Z}$. Consider the topological group relation quiver $Q_{n,m}(\mathbb{T})$ with $n\geq m > 0$. Then the subgroup $A$ of $\mathbb{T}$ defined by~\eqref{relsubgrp} is isomorphic to
\[
A=\{z\in\mathbb{T}~:~z^n=z^m\}=\{z\in\mathbb{T}~:~z^{n-m}=1\}\cong \mathbb{Z}_{n-m}.
\]

Since $\ker \beta\cong \mathbb Z_m$ is discrete, the right-invariant Haar measure $\mu$ is the counting measure, and $Q_{n,m}(\mathbb{T})$ is a topological graph. Let $\xi\in C\left(Q_{n,m}^1\right)$ be defined by $\xi(z,w)=\overline{z}w$, and let $\eta$ denote the identity function on $\mathbb{T}$ as an element of $C\left(Q_{n,m}^0\right)$. Suppose $X$ is the correspondence associated with $Q_{n,m}$ and $(\psi,\pi)$ is the universal covariant representation of $X$.

By~\cite[Corollary~3.310]{mccann}, the Cuntz-Pimsner algebra of $Q_{n,m}(\mathbb{T})$ is isomorphic to the unital universal $C^*$-algebra generated by the elements $U$ and $S$ satisfying the following relations:
\[
U^*U=UU^*=\sum_{k=0}^{n-1}U^kSS^*U^{-k}=1,\qquad S^*U^kS=\delta_{0,k},\quad 0\leq k\leq n-1,\qquad U^nS=SU^m.
\]
where $\delta_{n,m}$ is the Kronecker symbol. The isomorphism is given by $\psi(\xi)\mapsto S$ and $\pi(\eta)\mapsto U$. By \ref{ME}, $A$ acts on $C^*(S,U)$. From the action of $a\in A$ on $\xi$ and $\eta$, we deduce that $a$ fixes $S$ and $U\cdot a=aU$.
\end{example}

\section{Topological Aspects}\label{topaspects}

Given locally compact Hausdorff spaces $T, S$ and $f\in C_0(T\times S)$, we denote for any $s\in S$ the function $f_s\in C_0(T)$ defined by $f_s(t) = f(t, s),$ and call $f_s$ the cross section along $s$.
\begin{lemma}\label{convergenceofcrosssections}
    Let $X,Y,Z$ be second countable locally compact Hausdorff spaces,
    $s:X\rightarrow Y$ and $q:Z\rightarrow Y$ continuous, and $\left(\lambda_{y}\right)_{y\in Y}$
    an $s$-system. Suppose $f\in C_c(X\times Z)$ and  $v_{i}$ is a net
    in $Z$ converging to $v\in Z$. Then
    \[
    \int f_{v_i}(a)\:d\lambda_{q\left(v_{i}\right)}(a)\rightarrow\int f_v\left(a\right)\:d\lambda_{q(v)}(a)
    \]
\end{lemma}
\begin{proof}
    Let $\varepsilon>0$ be given, $v_i\to v$,  and $K=\text{\ensuremath{\pi_{1}(}supp}(f))$.
    By continuity of $q$ and $\lambda$, choose a neighborhood $V$ of $v$ such that
    if $u\in V$,
    \[
    \lambda_{q\left(u\right)}\left(s^{-1}\left(q\left(u\right)\right)\cap K\right)=\int_{K}1d\lambda_{q\left(u\right)}<\int_{K}1d\lambda_{q(v)}+1
    \]
    Set $M$ to be the right hand side of the expression above, and let $U$ be a neighborhood of $v$ such that for $u\in U$,
    \[
    \|f_u(a)-f_v(a)\|<\frac{\varepsilon}{M}
    \]
    Then for $v_{i}\in U\cap V$,
    \[
    \left|\int f_{v_{i}}\left(a\right)-f_v(a)\:d\lambda_{q\left(v_{i}\right)}(a)\right|<\frac{\varepsilon}{M}\int_{K}1\:d\lambda_{q\left(v_{i}\right)}<\varepsilon
    \]
    By continuity of $\left(\lambda_{b}\right)_{b\in B}$, we also have
    $\int f\left(a,v\right)\:d\lambda_{q\left(v_{i}\right)}(a)$ converging
    to $\int f\left(a,v\right)\:d\lambda_{q\left(v\right)}(a)$.
\end{proof}

The following theorem is the fundamental objective of this section. 
\begin{theorem}\label{Topologicalclassification}
Let $(F^{0},F^{1},s,r,\lambda)$ be a topological quiver and $G$
a locally compact group. For every principal $G$-bundle $q:P\rightarrow F^{0}$
with an isomorphism of pullbacks $\theta:s^{*}(P)\rightarrow r^{*}(P)$,
there exists a topological quiver $(E^{0},E^{1},\widetilde{s},\widetilde{r},\widetilde{\lambda})$
with $E^{0}=P$, $E^{1}=s^{*}(P)$ and a free and proper action of
G. Moreover,
every topological quiver $R = (E^{0},E^{1},\widetilde{s},\widetilde{r},\widetilde{\lambda})$
on which G acts freely and properly is isomorphic to one arising in this way.
\end{theorem}

\begin{proof}
Let $P$ and $q$ be given such that $q:P\rightarrow F^{0}$
is a principal $G$-bundle, and define $E^{1}:=s^{*}(P)$, $\widetilde{s}:=\pi_{2}$
as in the statement of the theorem. $\widetilde{s}$ is an open map since $E^{1}$ inherits the relative
topology from $F^{1}\times E^{0}$ and $\pi_{2}:F^{1}\times E^{0}\rightarrow E^{0}$
is open. We define $\widetilde{r}:=\pi_{2}\circ\theta$, and summarize our notation in the following commutative diagram

% https://q.uiver.app/#q=WzAsNyxbMCw0LCJGXjAiXSxbOCw0LCJGXjAiXSxbNCw0LCJGXjEiXSxbOCwyLCJFXjAgPVAiXSxbMCwyLCJFXjA9UCJdLFsyLDAsInJeKihQKSJdLFs2LDAsIkVeMT1zXiooUCkiXSxbMiwxLCJzIiwyXSxbMiwwLCJyIl0sWzMsMSwicSJdLFs0LDAsInEiLDJdLFs2LDMsIlxcd2lkZXRpbGRle3N9PVxccGlfMiJdLFs1LDQsIlxccGlfMiIsMl0sWzYsNSwiXFx0aGV0YSAiLDJdLFs2LDIsIlxccGlfMSJdLFs1LDIsIlxccGlfMSIsMl0sWzYsNCwiXFx3aWRldGlsZGV7cn0iXSxbNiw1LCJcXGNvbmciXV0=
\[\begin{tikzcd}[column sep=scriptsize] 	&& {r^*(P)} &&&& {E^1=s^*(P)} \\ 	\\ 	{E^0=P} &&&&&&&& {E^0 =P} \\ 	\\ 	{F^0} &&&& {F^1} &&&& {F^0.} 	\arrow["s"', from=5-5, to=5-9] 	\arrow["r", from=5-5, to=5-1] 	\arrow["q", from=3-9, to=5-9] 	\arrow["q"', from=3-1, to=5-1] 	\arrow["{\widetilde{s}=\pi_2}", from=1-7, to=3-9] 	\arrow["{\pi_2}"', from=1-3, to=3-1] 	\arrow["{\theta }"', from=1-7, to=1-3] 	\arrow["{\pi_1}", from=1-7, to=5-5] 	\arrow["{\pi_1}"', from=1-3, to=5-5] 	\arrow["{\widetilde{r}}", from=1-7, to=3-1] 	\arrow["\cong", from=1-7, to=1-3] \end{tikzcd}\]

We have a topological quiver $Q=(E^{0},E^{1},\widetilde{s},\widetilde{r},\widetilde{\lambda})$
if we can construct an $\widetilde{s}$-system of measures $\widetilde{\lambda}$.
Fix $v\in E^0$ and $\phi\in C_c(E^1)$ and then denote the cross section of $\phi$ along $v$ by $\phi_v\in C_c(F^1)$, so that $\phi_v(e) = \phi(e, v).$
We then define a family of Radon Measures $\left(\widetilde{\lambda}_{v}\right)_{v\in E^{0}}$ by 
\[
\int\phi(e,w)\:d\widetilde{\lambda}_{v}(e,w)=\int\phi_v(e)\:d\lambda_{q(v)}(e),
\]
where $(e, w)\in E^1\subseteq F^1\times E^0$. 
%Fix $v\in E^{0}$, and for $\phi\in C_{c}\left(E^{1}\right)$, define
%a family of Radon measures $\left(\widetilde{\lambda}_{v}\right)_{v\in E^{0}}$
%by
%\[
%\int\phi(e,w)\:d\widetilde{\lambda}_{v}(e,w)=\int\phi(e,v)\:d\lambda_{q(v)}(e)
%\]
%where $(e,w)\in E^{1}\subseteq F^{1}\times E^{0}$. 

Since $(d,u)\in\widetilde{s}^{-1}(v)$ implies $u=v$, and since $\phi(d, u)>0$ implies $\phi_u(d)$ has nonempty support in $s^{-1}\left(q(u)\right)$, we have $\text{supp}\left(\widetilde{\lambda}_{v}\right)$ is the full
preimage $\widetilde{s}^{-1}(v)$.
%If $(d,u)\in\widetilde{s}^{-1}(v)$
%and $\phi(d,u)>0$, then $u=v$ and the map $\phi_{0}$ defined by
%$\phi_{0}(e)=\phi(e,v)$ has non-empty support in $s^{-1}\left(q(v)\right)$.
%Thus $\text{supp}\left(\widetilde{\lambda}_{v}\right)$ is the full
%preimage $\widetilde{s}^{-1}(v)$. 
To see that the proposed family
of measures assembles continuously, let $\phi\in C_{c}\left(E^{1}\right)$,
and $v_{i}$ a net in $E^{0}=P$ which converges to $v$. Then by
the prior lemma \ref{convergenceofcrosssections},
%\[
%\int\phi(e)\:d\widetilde{\lambda}_{v}(e)-\int\phi(e')\:d\widetilde{\lambda}_{v_{i}}(e')
%\]
%\[
%=\int\phi(e,w)\:d\widetilde{\lambda}_{v}(e,w)-\int\phi(e',w')\:d\widetilde{\lambda}_{v_{i}}(e',w')
%\]
%\[
%=\int\phi(e,v)\:d\lambda_{q(v)}(e)-\int\phi\left(e',v_{i}\right)\:d\lambda_{q\left(v_{i}\right)}(e')
%\]

\begin{align*}
\int\phi(e,w)\:d\widetilde{\lambda}_{v}(e,w)&-\int\phi(e',w')\:d\widetilde{\lambda}_{v_{i}}(e',w')\\
=\int\phi_v(e)\:d\lambda_{q(v)}(e)&-\int\phi_{v_i}\left(e'\right)\:d\lambda_{q\left(v_{i}\right)}(e')
\end{align*}
converges to $0$.

Moreover, the quiver $Q$ has an action by $G$, where $G$ acts on an edge $(e,v)$
by $(e,v)\cdot g=(e,v\cdot g)$, and
%\[
%\int\phi(e,w\cdot g)\:d\widetilde{\lambda}_{v}(e,w)=\int\phi(e,v \cdot g)\:d\lambda_{q(v)}(e)
%\]
%\[
%=\int\phi(e,v \cdot g)\:d\lambda_{q(v\cdot g)}(e)=\int\phi(e,w)\;d\widetilde{\lambda}_{v\cdot g}(e,w)
%\]

\[\int\phi(e,w\cdot g)\:d\widetilde{\lambda}_{vg}(e,w)= \int \phi_{vg}(e)\,d\lambda_{q(vg)}(e)=\int \phi_v(e)\,d\lambda_{q(v)}(e)=\int \phi((e,w))\,d\widetilde{\lambda}_{v}(e,w).
\]
Therefore $\left(\widetilde{\lambda}_{v}\right)_{v\in E^{0}}$ is
a $G$-invariant $\widetilde{s}$-system of measures.

For the other direction, suppose $E=(E^{0},E^{1},\widetilde{s},\widetilde{r},\widetilde{\lambda})$
is a topological quiver with a free and proper action by $G$ such
that $\widetilde{\lambda}$ is $G$-invariant. Let $p_{i}:E^{i}\rightarrow F^{i}=E^{i}/G$
for $i\in\{0,1\}$ be the quotient maps. Equivariance of $\widetilde{s}$
and $\widetilde{r}$ induces natural maps $s:F^{1}\rightarrow F^{0}$
and $r:F^{1}\rightarrow F^{0}$. Further define a family of measures
$(\lambda_{w})_{w\in F^{0}}$ where for Borel $U\subseteq F^1$,
\[
\lambda_{p_{0}(v)}(U)=\widetilde{\lambda}_{v}\circ p_{1}^{-1}\left(U\cap s^{-1}\left(p_{0}(v)\right)\right)
\]

We must verify that this family is an $s$-system of measures. Recall that for any positive Radon measure $\mu$ on a locally compact Hausdorff space $X$,  $x\in \text{supp}(\mu)$ if and only if $\int \phi \,d\mu>0$ for every $\phi\in C_c(X, [0,1])$ such that $\phi(x)>0.$ This together with the isomorphism $E^1\cong s^*(E^0)$ coming from Theorem~\ref{thmhus} is enough to confirm that $p_0(v)\in \text{supp}(\lambda_{p_o(v)})$ if and only if $v\in \text{supp}(\widehat{\lambda}_v)$, from which it follows that each of the measures in $\lambda$ have full support. Regarding continuity, we first fix $\phi\in C_c(F^0)$ and consider $\psi = \phi\circ p_i\in C_b(E^1)$. Notice that $\psi$ is invariant for the obvious extended $G$-action, so that 
\[\int \phi(e)\,d\lambda_v = \int \psi(e')\,d\widetilde{\lambda}_w(e') = \int \psi(e')\,d\widetilde{\lambda}_{g\cdot w}(e')\] 
for any $g\in G$, where $w$ is any representative in the total space of $p_0(w)$. It follows that 
\[\int \phi(e)\,d\lambda_v-\int\phi(e)\,d\lambda{v_i}(e) = \int\psi(e')\,d\widetilde{\lambda}_{w}(e')-\int \psi(e')\,d\widetilde{\lambda}_{w_i}(e')\]
where now $w, w_i$ are any representatives such that $p_0(w)=v, p_0(w_i)=v_i.$ By fixing $w$ and adjusting the $w_i$ in their orbits, we may arrange for the $w_i$ to converge to $w$ and then choose $\eta\in C_c(E^1)$ with $\eta|_{s^{-1}(w)} = \psi|_{s^{-1}(w)}$. By continuous assembly of $\widetilde{\lambda}$ and the preceding four equations, it follows that $\lambda$ is a continuous $s$-system of measures.  
\end{proof}

\section{Correspondences}\label{correspondences}

Given an action on a topological quiver, we want to see how this induces an action on the associated quiver correspondence.

\begin{proposition}
Let $Q=(E^0, E^1, r, s, \lambda)$ be a topological quiver, and suppose that it it admits a continous action by $G.$ Then the associated correspondence $(X, A)$ admits a group action as well. 
\end{proposition}
\begin{proof}
It is well known how an action of a group on a space induces an action on the associated function space. Let $A=C_0(E^0)$ and given $G\curvearrowright E^0,$ we consider $\alpha_t(f)(v) = f(v\cdot t),$ where $f\in A$. Likewise, we consider for $\xi\in C_c(E^1), \gamma_t(\xi)(e) = \xi(e\cdot t)$, and claim that this extends to an action on the correspondence $(X, A)$. First we check continuity of $\gamma.$

Let $(t_i)_{i\in I}$ a net in $G$ converging to $t$, and let $\xi\in C_c(E^1)$. We claim that $\gamma_{t_i}(\xi)$ converges uniformly to $\gamma_t(\xi)$. To this end, take compact neighborhoods $K, K'\subset G$, with $t\in K\subset K'$, and define an auxiliary function $\mu:G\to [0,1]$ such that $\mu|_K = 1,$ and $\supp(\mu)\subseteq K'.$ Let $L=\supp(\xi)$ a compact subset of $E^1$, and put $C=L\cdot K'$, which is a compact subset of $E^1$ since the action is continuous. We note that $L\cdot t\subset C$. So we define another auxiliary function $m\in C_c(E^1)$ with $m|_C=1$. For use later, notice that $\|m\|_X\geq \lambda_v(C)$ for every $v\in E^0.$ Now, define $\eta:E^1\times G\to \C$ by the equation $\eta(e, s) = m(e)\xi(e\cdot s) \mu(s)$. Then $\eta$ is compactly supported, and there is an $i_0\in I$ so that for $i>i_0, t_i\in K,$ whence $\eta(e, t_i) = \xi (e\cdot t_i)$. Applying \cite{hskew}*{Lemma 3.1}, we assure that $\gamma_{t_i}(\xi)$ converges uniformly to $\gamma_t(\xi).$ 

So let $0<\epsilon<1$, and choose $i_0$ so that for $i>i_0, \|\gamma_t(\xi)-\gamma_{t_i}(\xi)\|_u<\frac{\epsilon}{\|m\|}.$

Then
\begin{align*}
\|\gamma_t(\xi)-\gamma_{t_i}(\xi)\|_X^2 &= \sup_{v\in E^0} \left \{\int |\xi(e\cdot t) - \xi(e\cdot t_{i})|^2\,d\lambda_v(e) \right \} < \frac{\epsilon^2}{\|m\|} \sup_{v\in E^0} \left \{\int \,d\lambda_v \right \} <\epsilon^2 \\
\end{align*}

Taking square roots, we see that $\gamma$ is continuous for compactly supported functions $\xi\in C_c(E^1)$. By density of this subspace in $X$, we have continuity of $\gamma$ everywhere. 

It is now easy to verify that the actions $\alpha$ and $\gamma$ are compatible with the correspondence structures. Indeed, for $\xi, \eta\in C_c(E^1)$ and $f\in C_0(E^0)$ 
\begin{align*}
\alpha_t(\langle \eta, \xi\rangle)(v) = \int \overline{\eta(e)}\xi(e) \,d\lambda_{v\cdot t} (e) = \int \overline{\eta(e\cdot t)} \xi(e\cdot t)\,d\lambda_v(e)
= \langle \gamma_t(\eta), \gamma_t(\xi) \rangle(v)\\
\gamma_t(\xi\cdot f)(e) = (\xi\cdot f)(e\cdot t) = \xi(e\cdot t)f(s(e\cdot t)) = \xi(e\cdot t)f(s(e)\cdot t) = (\gamma_t(\xi)\cdot \alpha_t(f))(e)\\
\gamma_t(f\cdot \xi)(e) = (f\cdot \xi)(e\cdot t) = f(r(e\cdot t))\xi(e\cdot t) = f(r(e)\cdot t)\xi(e\cdot t) = (\alpha_t(f)\cdot \gamma_t(\xi))(e),
\end{align*}
Where the equalities follow from the equivariance of the source and range maps.

\end{proof}

\begin{theorem}\label{ME}
If a locally compact group $G$ acts freely and properly on a topological quiver $Q,$ then $\cs(Q)\rtimes_r G$ and $\cs(Q/G)$ are strongly Morita equivalent. 
\end{theorem}
\begin{remark}
We will proceed with the following strategy, modeled after \cite{dkq}*{Section 5}. First, we will show that there is an action $\alpha$ of $G$ on $\cs(Q)$ which is saturated and proper in the sense of Rieffel \cite{r_proper}. Then we will construct an isomorphism between Rieffel's generalized fixed-point algebra $\cs(E)^\alpha$and the quotient quiver algebra $\cs(Q/G),$ and appeal to Rieffel's imprimitivity result. The strategy is summarized mathematically by the string of equivalences  
\[\cs(Q/G)\cong \cs(Q)^\alpha \cong_{\text{ME}}\cs(Q)\rtimes_r G,\]
where the second (Strong Morita) equivalence comes from \cite{r_proper}. Rieffel's generalized fixed-point algebra is essentially designed so that the Morita equivalence holds, but as in \cite{dkq}, we will find it easier to identify the fixed point algebra using the characterization provided by \cite{kqr_proper}*{Proposition 3.1}. The details of this identification are well summarized in \cite{dkq}*{Definition 5.5 } and the successive discussion, and with more detail in \cite{kqr_proper}. Throughout this article, we will specify to our setting, but the ideas discussed here work in complete generality.

The bulk of the work is in constructing the isomorphism mentioned above. To do this, we appeal to \cite{kqrfunctor} by constructing a correspondence homomorphism $\psi_\pi:(X(Q/G), A(E/G))\to (X(Q), A(Q)),$ proving that it is Cuntz-Pimsner covariant, appealing to \cite{kqrfunctor}*{Corollary 3.6}, and then identifying Rieffel's fixed point algebra as the image of the associated map. First, we offer a few preparatory techinical lemmas for topological quivers which \cite{dkq} consider for the graph setting.  
\end{remark}

\begin{lemma}\label{convergence}
Let $\Q = \{F^0, F^1, r, s, \lambda\}$ be a topological quiver, and $(X, A)$ be the associated correspondence. If $K\subset F^1$ is compact and  $f_i$ is a net in $C_K(F^1)$ converging uniformly to $f\in C_K(F^1),$ then $f_i$ coverges to $f$ in the associated correspondence norm $\|\cdot\|_X$.
\end{lemma}

\begin{proof}
Take $K'\subseteq F^1$ with $K\subset K',$ $K'$ compact, and let $h\in C_{c}(F^1)$ with $h|_K = 1, \|h\|_u<1, h\geq 0,$ and $\supp(h)\subseteq K'.$ Then Define $H:F^0\to \C$ by $H(v) = \int h(e)\,d\lambda_v(e)$, which is a continuous function with support confined to the compact set $s(K')$. Now, we compute 
\begin{align*}
\|f-f_i\|^2_X &= \langle f-f_i, f-f_i\rangle_X  = \sup_{v\in F^0} \int |f(e)-f_i(e)|^2\,d\lambda_v(e)\\
&\leq \sup_{v\in F^0} \int \|f-f_i\|^2_u h(e)\, d\lambda_v(e)\leq \|f-f_i\|^2_u\cdot \|H\|_u
\end{align*}
Since $H$ is a continuous compactly supported function, there is an $M$ such that $M\geq \|H\|_u$, showing that uniform convergence of the $f_i$ implies convergence in the correspondence norm. 
\end{proof}

In \cite{dkq}*{Corollary 3.9} the authors offer a result with two parts. While the second statement fails for topological quivers, the first one is retained for exactly the same reasons. For convenience and clarity, we restate the result relevant to us here. 

\begin{lemma}{\cite{dkq}*{Corollary 3.9(a)}}\label{keepingcompact}\\
Let $\Q = (F^0, F^1, r, s, \lambda)$ be a topological quiver admitting a principal action by locally compact group $G,$ and let $\R = (E^0, E^1, \dot{r}, \dot{s}, \lambda)$ be the associated quotient. Let $K\subset E^1$ be compact. 
Then for every compact $L\subset F^0$ the sets $r\inv(L)\cap q\inv(K)$ and $s\inv(L)\cap q\inv(K)$ are compact in $F^1.$
\end{lemma}

\begin{proof}
As in the topological graph setting, this is a direct consequence of \cite{dkq}*{Lemma 3.8}, which ultimately comes from the isomorphism of Theorem \ref{thmhus}.  
\end{proof}

\begin{proposition}\label{homomorphism}
There is an injective $*$-homomorphism $\Pi:\cs(Q/G)\to M(\cs(Q))$. 
\end{proposition}

\begin{proof}
We begin by fixing notation. Let $A=C_0(E^0/G), X=\overline{C_c(E^1/G)}^{\|\cdot\|}, B=C_0(E^0), \text{ and } Y=\overline{C_c(E^1)}^{\|\cdot\|}.$ Recall that $M(Y)=\mathcal{L}(B, Y)$, and we must construct a homomorphism $(\mu, \nu):(X, A)\to (M(Y), M(B) )$. The upshot is that both of these maps arise as precomposition with the quotient maps for the action of $G$ on $Q$. 

First, notice that all of  content of \cite{hskew}*{Lemma 4.6-Proposition 4.10} holds with minor cosmetic changes to reflect the absence of global triviality in the present setting. For completeness, we will abridge the arguments given there and highlight the required changes. Fixing $\xi \in C_c(E^1/G)$ and $f\in C_c(E^0)$, the function 
\[(\xi\circ q )\cdot (f\circ s):F^1\to \C\]
has support confined to the compact set $(\supp(\xi)\times \supp(f)) \cap (E^1/G*_{\dot{s}}E^0 )$ since $F^1$ is isomorphic to the pullback. Thus, the map 
\begin{align*}
\mu(\xi):C_c(E^0)&\to C_c(E^1)\\
f&\mapsto (\xi\circ q)\cdot (f\circ s) 
\end{align*}
is well defined. This map is linear, and the computation 
\begin{equation}\label{rtaction}
\langle \mu(\xi)\cdot f, \mu(\xi)\cdot f\rangle(v) = \langle \xi, \xi\rangle(q(v))|f(v)|^2\leq \|\xi\|^2\|f\|^2,
\end{equation}
shows that $\mu(\xi)$ is bounded and extends to a bimodule map from $B\to Y$. This map admits an adjoint $\mu(\xi)^*:Y\to B$ given for $\eta\in C_c(E^1)$ by 
\[\mu(\xi)^*\eta (v) = \int \overline{\xi\circ q(e)} \eta(e)\,d\lambda_v(e).\]
Restricting $\eta$ and $\xi\circ q$ to $s^{-1}(v),$ we have 
\[\mu(\xi)\eta (v) = \langle \xi\circ q, \eta \rangle\in \C, \]
the value of the inner product for $L^2(s^{-1}(v), \lambda_v)$. We mention also here that, by \cite{husemoller}*{Chapter 1, Proposition 5.5}, the function $\xi\circ q$ restricted to $s^{-1}(v)$ agrees with $\xi$ restricted to $s^{-1}(q(v )).$ Boundedness of this map now follows from the Cauchy-Schwarz inequality, justifying it's extension to a map $\mu(\xi)^*: Y \to B$. That it is an adjoint for $\mu(\xi)$ follows easily from the definitions. 

To confirm boundedness of the left action $f\cdot \mu(\xi) = (f\circ r)\cdot (\xi\circ q)$, we verify 
\begin{align}\label{ltaction}
\langle f\cdot \mu(\xi), f\cdot \mu(\xi) \rangle (v) &=\int |f\circ r|^2|\xi\circ q|^2 \, d\lambda_v\\
&\leq \|f\|^2\langle \xi, \xi\rangle(q(v))\leq \|\xi\|^2\|f\|^2_u.
\end{align}
Then, the covariance of the left action and inner-product compatibility follows from the commutativity of the range and source maps with the quotient maps as in \cite{hskew}*{Proposition 4.10}. 

The above computations \ref{rtaction} and \ref{ltaction} show that, in fact, $\mu:X\to M_B(Y),$ since the integrands are compactly supported functions. Additionally, $\nu:A\to M(B)$ is nondegenerate as precomposition by continuous functions always is. Thus, it remains to verify axioms (iii) and (iv) of Definition \ref{CPcovcor}, which are more delicate than in the setting of \cite{hskew}. To begin, we must verify that $\nu$ maps the Katsura ideal $J_X$ into $M(B;J_Y) = \{m\in M(B)|\, mB\cup Bm\subset J_Y\}$. For topological quivers, the Katsura ideal is characterized as the ideal of functions supported on the \emph{regular vertices} $E^0_{\text{reg}} = E^0_{\text{fin}}\cap \text{Int}(\overline{r(E^1)}),$ where 

\begin{align*}
E^0_{\text{fin}} = \{ v\in E^1:\, &\text{there exists a precompact neighborhood } V \text{ of } v \text{ so that }r\inv(\overline{V})\\
&\text{is compact and } s|_{r\inv(V)} \text{ is a local homeomorphism}\}.
\end{align*}

To prove that $\nu(J_X)\subseteq M(B;J_Y)$, we have to show that if $f\in C_0(F^0_{\rm reg})$ and $g\in C_0(E^0)$, then $\nu(f)g=(f\circ q^0)g\in C_0(E^0_{\rm reg})$. This is equivalent to the implication
\[
(\nu(f)g)(v)\neq 0\quad\Rightarrow\quad v\in E^0_{\rm reg}.
\]
First, we show that $v\in E^0_{\rm fin}$. If $(\nu(f)g)(v)\neq 0$, then $\nu(f)\neq 0$, which, in turn, implies that $q^0(v)\in F^0_{\rm reg}$. By definition, there exists a precompact neighborhood $V$ of $q(v)$ such that $r_F^{-1}(\overline{V})$ is compact and $s_F|_{r_F^{-1}(V)}$ is a local homeomorphism. Since $(q^0)^{-1}(V)$ is open and $E^0$ is locally compact Hausdorff, there exist a precompact neighborhood $U$ of $v$ such that
\[
U\subseteq\overline{U}\subseteq (q^0)^{-1}(V)\subseteq (q^0)^{-1}(\overline{V}).
\]
We will show the first condition, that $r_E^{-1}(\overline{U})\subseteq E^1$ is compact. Recall that, by $G$-equivariance of $r_E$, $E^1$ is homeomorphic to the total space of the pullback bundle $r_E^*(E^0)=\{(x,u)\in F^1\times E^0~|~r_F(x)=q^0(v)\}$ by Theorem \ref{thmhus}. Next, since $\overline{U}\subseteq (q^0)^{-1}(\overline{V})$, note that
\[
r^{-1}_E(\overline{U})\subseteq r^{-1}_E((q^0)^{-1}(\overline{V}))=(q^0\circ r_E)^{-1}(\overline{V})=(r_F\circ q^1)^{-1}(\overline{V})=(q^1)^{-1}(r_E^{-1}(\overline{V})).
\]
Therefore, we obtain that
\[
r_E^{-1}(\overline{U})\subseteq (r_E^{-1}(\overline{V})\times\overline{U})\cap r^*_E(E^0).
\]
Since the right-hand side is compact, we conclude that $r_E^{-1}(\overline{U})$ is compact.

It remains to prove that $s_E|_{r^{-1}_E(U)}$ is a local homeomorphism. Since $s_E$ and $r_E^{-1}(U)$ are open, it suffices to show that $s_E:r^{-1}_E(U)\to E^0$ is locally injective. Recall that, by $G$-equivariance of $s_E$, the space $E^1$ is homeomorphic to the total space of the pullback bundle $s_E^*(E^0)$. Since $s_F:r_F^{-1}(V)\to F^0$ is locally injective, there is an open set $V'$ in $r_F^{-1}(V)$ such that $s_F|_{V'}$ is injective. Consider the open set
\[
U':=(s^{-1}_F(V')\times U)\cap s^*_E(E^0)\cap r_E^{-1}(U).
\]
Assume that $s_E(x_1,u_1)=s_E(x_2,u_2)$ for some $(x_1,u_1),(x_2,u_2)\in U'$. First, since $s_E(x,u)=u$, we obtain that $u_1=u_2$. Furthermore, $x_1,x_2\in V'$ and 
\[
s_F(x_1)=u_1=u_2=s_F(x_2),
\]
so we conclude that $x_1=x_2$.

To prove that $v\in {\rm Int}(\overline{r_E(E^1)})$, we have to find an open set $U\subseteq \overline{r_E(E^1)}$ such that $v\in U$. 
First, since $q^0(v)\in F^0_{\rm reg}\subseteq {\rm Int}(\overline{r_F(F^1)})$, there exists an open set $V\subseteq \overline{r_F(F^1)}$ such that $q(v)\in V$. Now, $v$ belongs to an open set $U:=(q^0)^{-1}(V)\subseteq (q^0)^{-1}(\overline{r_F(F^1)})$. For any $u\in U$, there is a net $(x_i)_{i\in I}\subseteq F^1$ such that $q^0(u)=\lim_{i\in I} r_F(x_i)$. Since $q^0$ is a quotient open map, we can find a net $(u_j)_{j\in J}\subseteq E^0$ such that $\lim_{j\in J} u_j=u$ and $q^0((u_j)_{j\in J})\subseteq (r_F(x_i))_{i\in I}$, i.e. there is a subnet $(r_F(x_{i_j}))_{j\in J}$ such that $q^0(u_j)=r_F(x_{i_j})$ for all $j\in J$. Consequently, since $E^1\cong r^*_E(E^0)$, we conclude that $(x_{i_j},u_j)\in E^1$ for all $j\in J$ and $\lim_{j\in J}r_E(x_{i_j},u_j))=\lim_{j\in J}u_j=u$. Therefore, $U\subseteq \overline{r_E(E^1)}$ and we are done.

So $v$ is a regular vertex of $B$, and it follows that $\nu(J_X)\subseteq M(B; J_Y)$. 

It remains to show the final condition of Definition \ref{CPcovcor}. For this, there are two essential computations which differ from \cite{hskew}*{Lemma 4.12}. Taking $\eta, \xi \in C_c(E^1)$ and $y\in C_c(F^1)$, one sees that for arbitrary but fixed $(e,v)\in s^*(P)$
\begin{align*}
\theta_{\eta, \xi}(y) (e, v) &= \mu(\eta)\mu(\xi)^*(y)(e, v) = \eta\circ \pi_1(e, v)\cdot \int \overline{\xi\circ \pi_1 (e, v)}y(e, v)\,d\lambda_{(e, v)}\\
&= \eta(e)\cdot \int \overline{\xi(e)} y(e, v)\,d\lambda_v = [\eta\cdot \langle \xi, y_v\rangle](e).
\end{align*}
On the other hand, for $f\in J_X$ we have 
\begin{align*}
\overline{\phi_b}\circ \nu(f) y(e, v) &= [\overline{\phi_B}|(f\circ q) y](e, v) = f\circ q\circ \tilde{r}(e, v) y(e, v)\\
&= f\circ r\circ \pi_1(e, v) y(e, v) = [\phi_A(f)\cdot y(\cdot, v)](e).
\end{align*}
With these two computations replacing those which appear in \cite{hskew}*{Lemma 4.12}, the arguments are identical. This establishes commutativity of the diagram, and hence Cuntz-Pimsner covariance of the correspondence homomorphism $(\mu, \nu)$. 

The homomorphism $\pi: \cs(Q/G)\to \cs(Q)$ is now a consequence of \cite{kqrfunctor}*{Corollary 3.6}. Since $\nu$ is injective, $\Pi$ is too. 
\end{proof}

We now proceed with the proof of the main Theorem \ref{ME} establishing a Morita Equivalence between the reduced crossed product and the \cs-algebra associated to the quotient quiver.
\begin{proof}
That the action is saturated and proper follows from \cite{hrw_symm}*{Lemma 4.1}; all that is required is a free and proper action of a group on the topological space $E^0$, an action on $\cs(Q),$ and an equivariant homomorphism from $C_0(E^0)$ to $M(\cs(Q))$ where $C_0(E^0)$ is given the translation action. The canonical representation of $C_0(E^0)$ and the action $\alpha$ on $\cs(Q)$ induced from the associated correspondence action supply this. Thus, we have
\begin{itemize}
\item Rieffel's generalized fixed-point algebra $\cs(Q)^\alpha$;
\item the dense $*$-subalgebra in $\cs(Q)$
\[\cs(Q)_0 = \text{span}\{ \,k_A(C_0(E^0)) \cs(Q) k_A(C_0(E^0))\},\]
and 
\item a linear map $\Phi:\cs(Q)_0\to \{a\in M(\cs(Q)): \overline{\alpha_g}(a) = a \text{ for all } g\in G \}$.
\end{itemize}

Kaliszewski, Quigg, and Raeburn show that the image of $\Phi$ is a $*$-subalgebra of $M(\cs(Q))$ \cite{kqr_proper}*{Proposition 3.1}, so that it's norm closure is a \cs-algebra, and that this \cs-algebra coincides with $\cs(Q)^\alpha$. So it is enough to show that 
\[\cs(Q)^\alpha\subseteq \Pi(\cs(Q/G)) \text{ and } \Pi(\cs(Q/G))\subseteq \cs(Q)^\alpha\]
through this identification. 

We begin by noting that the for any topological quiver $R = (E^0, E^1, r, s, \lambda)$, space $E^n$ of paths of length $n$ is a topological quiver when endowed with the product topology and an appropriate system of measures \cite{mtquiver}.  Then $X(E^n)$ is isomorphic as a $C_0(E^0)$-correspondence to the $n$-fold balanced tensor product 
\[X(R)^n=\underbrace{X(R)\otimes_{C_0(E^0)}\cdots \otimes_{C_0(E^0)} X(R)}_{n-\text{factors }}.\]
Moreover, if $R$ admits a principal action by $G,$ so does the quiver $R^n$ whose edge space comprises paths of length $n$ (the action on each factor commutes with the range and source maps). By the universal property of $\cs(R),$ there are canonical maps $(k_{X(R)}^n, k_{C_0(E^0)}):(X(R)^n, C_0(E^0))\to \cs(R)$.
Moreover, since $\cs(R)$ is realized as the coisometric representation on the Fock Space of $X(R)$, we have that
\[\cs(R) = \overline{
\text{span} }\{k_{X(R)}^n(C_c(E^n))k_{X(R)}^m(C_c(E^m))^*: m, n \in \mathbb{N}\}.\]

Now we specialize to our setting. Let $Q = (E^0, E^1, r, s, \lambda)$ be a topological quiver admitting a principal action by locally compact group $G,$ and resume the notation from the proof of Proposition \ref{homomorphism}: $A=C_0(E^0/G), X=\overline{C_c(E^1/G)};$ and $B=C_0(E^0), Y = \overline{C_c(E^1)}.$ By abuse of notation, we write $\tau^n = \overline{k_Y^n}\circ \mu$, where $\mu$ is precomposition by the quotient map for the action of $G$ on $E^n$ (instead of on $E^1$).

We now begin with the first inclusion $\cs(Q)^\alpha\subseteq \Pi(\cs(Q/G))$. According to our alternative description of $\cs(Q)^\alpha,$ we need only consider $f, h\in C_c(E^0)$ and $a\in \cs(Q)$ and show for a generic element  $k_B(f)ak_B(h)$ of $\cs(Q)_0,$ that 
\[\Phi(k_B(f)ak_B(h) ) \in \Pi(\cs(Q)).\] 
According to the previous discussion, we may further specialize so that the central factor is of the form $a = k_Y^n(\xi)k_Y^m(\eta)^*$, where $\xi\in C_c(E^n)$ and $\eta\in C_c(E^m)$. Then, we  apply the arguments of \cite{dkq}*{Theorem 5.6, step 5} to yield the inclusion. This entails verification that the product $k_B(f)k_Y^n(\xi)k_Y^m(\eta)^*$ approximately factorizes  as an element $bk_B(h)$ whose image under $\Pi$ is in the span of elements of the form $k^n_Y(\xi_1)k^m_Y(\eta_1)^*\Phi(k_B(h))$ for $\xi_1, \eta_1$ suitably compactly supported functions in the $n$ and $m$-fold tensor products respectively (see \cite{dkq}*{page 1552}). This is accomplished by an appeal to \cite{dkq}*{Lemma 5.8}, with a slight modification in the proof-- where they appeal to \cite{dkq}*{Lemma 5.7}, we instead appeal to our own Lemma \ref{convergence}. 

The other inclusion, 
\[\Pi(\cs(Q/G))\subseteq \cs(Q)^\alpha,\]
follows along similar lines. The arguments from \cite{dkq}*{Theorem 5.6} continue to apply, and \cite{dkq}*{Lemma 5.9} holds for topological quivers without change. The only point of ambiguity is in their appeal to \cite{dkq}*{Corollary 3.9}, but the result they use is more precisely Corollary 3.9(a), which we parse for topological quivers in Lemma \ref{keepingcompact} above. 
\end{proof}

We'll conclude this paper with a couple of applications to the previous theorem, computing Morita equivalences for a handful of examples from section \ref{actionexample}. 
\begin{example}
We'll first consider the coset quivers associated to Euclidean groups $E(n)$. As computed above, the quotient quiver (which is indeed a quiver) is $Q_{\mathbb{R}^n<\mathbb{R}^n}^1$, so we have 
\[\cs(Q_{\mathbb{R}^n<E(n)}^k)\rtimes_r O(n) \cong \cs(Q_{\mathbb{R}^n<\mathbb{R}^n}^1) = \mathcal{O}_\infty.\]

Of course, as we noticed in example \ref{euclidean}, the coset quiver is actually a skew product, so this computation is also available using the techniques of \cite{hskew}.
\end{example}

\begin{example}
Next, we'll compute the Morita equivalence for Example \ref{qr}. As we saw, $Q_\mathbb{R}$ admits a $\mathbb{Z}$ action for which the associated bundle is nontrivial-- hence $Q_\mathbb{R}$ is not a skew product. We have 
\[\cs(Q_\mathbb{R})\rtimes_r \mathbb{Z}\cong \cs(Q_\mathbb{T}),\]
which we've already determined as a homeomorphism algebra. Since $Q_\mathbb{R}$ is a topological graph, this example reduces to the results of \cite{dkq}.
\end{example}
\begin{example}
For a final application, we'll compute the Morita equivalence of Example \ref{kickass}. Choosing groups $N, K, H$ together with the proposed $H$-equivariant map $\phi:N\rtimes H\to N\rtimes H,$ we find that 
\[\cs(Q_{N<N\rtimes K}^\phi)\rtimes _r H\cong \cs(Q(N, K , H, \dot{\phi})),\]
where the quiver on the right is an abreviateion for the quotient quiver appearing in Example \ref{kickass}. As mentioned there, $Q_{N<N\rtimes K}^\phi$ is a genuine topological quiver and cannot be realized as a skew product. Thus, the Morita equivalence computed above is only accessible using the techniques from this work. 
\end{example}

\end{document}